\documentclass[12pt]{amsart}
\usepackage{amsmath}
\usepackage{amsfonts}
\usepackage{amssymb}
\usepackage{amsthm}
\usepackage{graphicx}
\usepackage{hyperref}
\hypersetup{
  colorlinks=true,
  linkcolor=blue,
  urlcolor=black,
  citecolor=blue
}
\usepackage{color}
\usepackage[margin=1in]{geometry}

\newtheorem*{theorem*}{Theorem}
\newtheorem{Theorem}{Theorem}[section]
\newtheorem{Corollary}[Theorem]{Corollary}
\newtheorem{Lemma}[Theorem]{Lemma}
\newtheorem{Definition}[Theorem]{Definition}
\newtheorem{Proposition}[Theorem]{Proposition}
\newtheorem{Remark}[Theorem]{Remark}
\newtheorem{Example}[Theorem]{Example}

\newcommand{\osc}{\mathrm{osc}}

\newcommand{\Lop}{\mathcal{L}}
\newcommand{\R}{{\mathbb R}}

\newcommand{\N}{{\mathbb N}}

\newcommand{\J}{\mathcal{J}}
\newcommand{\la}{\left\langle}
\newcommand{\ra}{\right\rangle}
\def \tr {\mathrm{tr}}

\def \H {{\mathbb H}}

\def \Hn {{\mathbb H}^n}
\def \L {\mathcal{L}}
\def \e {\epsilon}
\def \O {\Omega}

\def \r {\mathcal R}

\def \a {\alpha}
\def \t {\tau}
\def \la {\left\langle }
\def \ra {\right\rangle}
\def \tr {\mathrm{tr}}

\begin{document}

\title[Boundary regularity for subelliptic equations in $\mathbb{H}^n$]{
Boundary regularity for subelliptic equations\\ in the Heisenberg group}
\author[F. Abedin and G. Tralli]{Farhan Abedin and Giulio Tralli} 
\address{Department of Mathematics, Lafayette College, Easton, PA 18042}
\email{abedinf@lafayette.edu}
\address{Dipartimento di Matematica e Informatica, Universit\`a di Ferrara, Via Machiavelli 30, 44121 Ferrara, Italy}
\email{giulio.tralli@unife.it}

\date{\today}

\begin{abstract}
We prove boundary H\"older and Lipschitz regularity for a class of degenerate elliptic, second order, inhomogeneous equations in non-divergence form structured on the left-invariant vector fields of the Heisenberg group. Our focus is on the case of operators with bounded and measurable coefficients and bounded right-hand side; when necessary, we impose a dimensional restriction on the ellipticity ratio and a growth rate for the source term near characteristic points of the boundary. For solutions in the characteristic half-space $\{t>0\}$, we obtain an intrinsic second order expansion near the origin when the source term belongs to an appropriate weighted $L^{\infty}$ space; this is a new result even for the frequently studied sub-Laplacian.
\end{abstract}

\maketitle


\section{Introduction}


The regularity of solutions of uniformly elliptic second order equations in non-divergence form with bounded and measurable coefficients has been an active field of study for several decades. Some of the landmark achievements in this area are the scale-invariant Harnack inequality and interior H\"older estimates, which were proved using the powerful techniques introduced by Krylov-Safonov \cite{KrylovSafonov80} for strong solutions of linear equations and Caffarelli \cite{Caffarelli89} for viscosity solutions of fully nonlinear equations. The Krylov-Safonov approach can be traced back to earlier influential work by Landis \cite{LandisDokl, LandisBook}.

As for boundary regularity, classical barrier arguments yield the Hopf lemma and Lipschitz estimates at the boundary when the domain satisfies interior and exterior sphere conditions and the boundary data has sufficient regularity; we refer the reader to the survey \cite{ApushkinskayaNazarovSurvey} for appropriate references. Work of Miller \cite{Miller67} and Michael \cite{
Michael81} establishes boundary H\"older regularity results in domains satisfying an exterior cone condition through the construction of barriers adapted to the boundary geometry; see also \cite{ChoSafonov, Safonov2018} for various generalizations. An important H\"older estimate for the normal derivative was obtained by Krylov \cite{Krylov83} to establish solvability of the Dirichlet problem for fully nonlinear equations.

The aforementioned works have spurred the development of analogous regularity results for equations that are not uniformly elliptic. The literature encompassing such results is vast, and we will not attempt to survey this immense body of work here. We simply note that the precise manner in which the ellipticity becomes degenerate/singular necessarily influences the approach to regularity.

In this paper, our focus is on a class of degenerate elliptic equations in non-divergence form structured on the left-invariant vector fields of the Heisenberg group, one of the prototypical non-Abelian Lie groups. Specifically, given a matrix field $A : \R^{2n+1} \to \R^{2n \times 2n}$ satisfying the uniform ellipticity condition
\begin{equation}\label{ellipticity-of-A}
0<\lambda \mathbb{I}_{2n}\leq A(x,t) \leq \Lambda \mathbb{I}_{2n} \qquad \text{for all } (x,t) \in \R^{2n+1},
\end{equation}
we will study solutions of the equation $\Lop_A u = f$ in a bounded domain $\O$, where $\L_A$ is defined in \eqref{operator-definition} and can be written as
\begin{equation}\label{L-in-coords}
\L_A = \tr \left(\mathcal{A} D^2 \cdot \right) \quad  
\mathcal{A}(x, t) = \left( \begin{array}{cc}
A(x,t) & 2 A(x,t) \J x \\
2 (A(x,t) \J x)^T & 4 \la A(x,t) \J x, \J x \ra \end{array} \right) \in \R^{(2n+1)\times(2n+1)},
\end{equation}
$D^2$ denotes the Hessian operator in $\R^{2n+1}$, and $\J$ is the standard $2n \times 2n$ symplectic matrix 
\begin{equation}\label{standard-symplectic}
\J := \left( \begin{array}{cc}
0 & -\mathbb{I}_n \\
\mathbb{I}_n & 0 \end{array} \right).
\end{equation}
The source term $f$ is assumed to belong in an appropriate subset of $L^{\infty}(\O)$. Our goal is to establish various \emph{a priori} estimates for $u$ and its derivatives near the boundary $\partial \O$. We note that even though the matrix $A$ is uniformly elliptic, the operator $\Lop_A$ is always degenerate elliptic. Indeed, it is straightforward to verify that for all $(x,t) \in \O$, 
\[
\langle\mathcal{A}(x,t) \zeta,\zeta\rangle = \langle A(x,t)(\xi + 2\t \J x), (\xi + 2\t \J x)\rangle \quad \text{for any } \zeta = (\xi, \t) \in \R^{2n} \times \R.
\]
Therefore, $\mathcal{A}(x,t)$ is non-negative definite for each $(x,t) \in \O$, but the kernel of $\mathcal{A}(x,t)$ is one-dimensional at each $(x,t) \in \O$ for any matrix field $A$ satisfying \eqref{ellipticity-of-A}.

The study of the regularity theory for operators similar to $\L_A$ was initiated in the seminal works of Kohn \cite{Kohn65}, Folland \cite{Folland73, Folland75}, Folland-Stein  \cite{FollandStein74}, and Jerison \cite{Jerison81-1, Jerison81-2}; we refer to the recent survey \cite{Folland2021} for a detailed historical account and motivating connections with CR geometry. Much is known about the sub-Laplacian (also known as the real part of the Kohn-Laplacian), which corresponds to $A = \mathbb{I}_{2n}$; see the monograph \cite{BLU2007}. For operators in divergence-form, the classical De Giorgi-Nash-Moser program has been successfully adapted to this degenerate setting. There has been some progress on developing a regularity theory for the non-divergence form operator $\L_A$ with minimal regularity hypotheses on the coefficient matrix $A$ \cite{AGT2017, GT2011, T2014}, but the biggest obstacle to proving a Caffarelli-Krylov-Safonov result for $\L_A$ is the lack of an appropriate Aleksandrov-Bakelman-Pucci (ABP) type maximum principle.  

Developing a satisfactory boundary regularity theory for $\L_A$ has also proved to be challenging due to subtle issues arising from the presence of so-called characteristic points on the boundary (in the sense of Fichera \cite{Fichera}), which are points where the normal vector $\nu$ of $\partial \O$ belongs to the kernel of the matrix $\mathcal{A}$. For a large family of degenerate-elliptic operators, boundary regularity at non-characteristic points is established in the classical works of Kohn-Nirenberg \cite{KohnNirenberg65} and Oleinik-Radkevi\v{c} \cite{OR}. In the special case of the Heisenberg group, a detailed investigation at non-characteristic points was carried out by Jerison \cite{Jerison81-1} for the sub-Laplacian. More recently, there has been a flurry of activity  \cite{BaldiCittiCupini2019, BanerjeeGarofaloMunive2019, BanerjeeGarofaloMunive-Annalen2024, Citti-Giovannardi-Sire} on obtaining sharp Schauder-type regularity results around non-characteristic boundary points under various assumptions on the operators and the domain boundary. The counterpart of such regularity results around characteristic boundary points is, at present, in an incomplete state. Both positive and negative results in this direction have been established by Jerison in the important work \cite{Jerison81-2}; we will elaborate upon this below. Potential theoretic techniques (see, e.g., \cite{Bony, BLU2007, LU2010}) have also proven to be quite effective in establishing boundary H\"older estimates at both characteristic and non-characteristic points for solutions to Dirichlet-type problems for a large class of equations, provided the coefficients of the operator has a modulus of continuity and the domain satisfies certain exterior metric/capacitary assumptions. Lipschitz estimates and Poisson-kernel bounds for $\Delta_X$-harmonic functions under exterior ball conditions were obtained in \cite{Lanconelli-Uguzzoni-PoissonKernel}, and in more general settings in \cite{CapognaGarofaloNhieu02}. The second author, in a previous work with Martino \cite{MartinoTralli2016}, has also established an analogue of the Hopf-Oleinik lemma for non-divergence form operators $\L_A$ with bounded and measurable coefficients at characteristic points of the boundary assuming the domain satisfies an interior touching ball condition.
 
\subsection{Main Results} Let us now discuss our results informally, making references to specific theorems in the body of the text for precise statements. We note that our regularity results are in terms of the metric $d$ compatible with the homogeneous group structure (see \eqref{defmetric}) and the geometric hypotheses on the domain boundary are with respect to metric balls $B_r$. Our approach is modeled on Landis-type growth lemmas (see Theorems \ref{growthlemmalandisH} and \ref{bogrth}) and hinges on comparison principle arguments. The constructions of the necessary barrier functions are tailored to the geometric properties of the boundary; this is captured by the definitions of the positive exterior density, exterior ball containment, and exterior touching ball conditions, which can be found in Sections \ref{sec:Holder} and \ref{sec:lipschitz}. Note that these geometric conditions do not preclude the existence of characteristic points at the boundary. 

For operators $\L_A$ satisfying a dimensional restriction on the ellipticity ratio $\frac{\Lambda}{\lambda}$ (see \eqref{CordesLandis} for the precise statement), we prove
\begin{enumerate}
\item[(i)] a scale-invariant inhomogeneous Harnack inequality, Theorem \ref{thm:Harnack};
\item[(ii)] uniform boundary H\"older estimates in domains satisfying a positive exterior density condition, Theorem \ref{thm:bdy-Holder};
\item[(iii)] a pointwise second order asymptotic expansion near the origin for solutions vanishing on the boundary of the characteristic half-space $\{t > 0\}$, Theorem \ref{ScaleInvariantHolderEstimateNormalDerivative}.
\end{enumerate}

The assumption \eqref{CordesLandis} made in the aforementioned results is sometimes referred to as a \emph{Cordes-Landis condition} (see \cite{T2014}) and, while admittedly restrictive, it is at the time of this writing the weakest hypothesis on the coefficients of $\L_A$ under which an interior Harnack inequality is known to hold for homogeneous equations. Indeed, the proof of the inhomogeneous growth lemma, Theorem \ref{growthlemmalandisH} below, which is a key ingredient in the proof of the inhomogeneous Harnack inequality, Theorem \ref{thm:Harnack}, relies on the fact that a certain barrier function \eqref{definitionofUE} is a subsolution when \eqref{CordesLandis} holds. As such, \eqref{CordesLandis} forms the bottleneck in our ability to obtain regularity results with no a priori restriction on the ellipticity ratio, which is a well known open problem in the field, even for homogeneous equations.

This makes our second set of results noteworthy, as they hold for operators $\L_A$ with bounded, measurable coefficients, but \emph{arbitrary} ellipticity ratio. We prove
\begin{enumerate}
\item[(iv)] uniform boundary H\"older regularity in domains satisfying an exterior ball containment condition, Theorem \ref{thm:bdy-Holder};
\item[(v)] pointwise boundary Lipschitz regularity in domains satisfying an exterior touching ball condition, Theorem \ref{thm:bdy-Lip};
\item[(vi)] a ``linear-in-$t$'' growth estimate for solutions vanishing on the boundary of the characteristic half-space $\{t > 0\}$, Theorem \ref{primadinocond}.
\end{enumerate}

The fact that we are able to dispense with the assumption \eqref{CordesLandis} in cases (iv)-(v)-(vi) is loosely related to the principle that solutions of elliptic equations behave better at the boundary than in the interior, as the boundary geometry and Dirichlet data can be exploited to create useful barriers. 

Let us comment further on an aspect that plays a pervasive role in this work, which is the behavior of the inhomogeneous source term $f$. In the results (i)-(ii)-(iv) listed above, which concern interior and boundary H\"older regularity of solutions, the term $f$ appears in the relevant estimates via the standard $L^\infty$-norm. On the other hand, when investigating boundary derivatives estimates, we are forced to make a more restrictive assumption on $f$, and so the estimates corresponding to the results (iii)-(v)-(vi) depend on a weighted $L^\infty$-norm (see Definition \ref{weightDef}). We note that this phenomenon is unrelated to the non-variational structure of $\Lop_A$ and the lack of regularity of its coefficients; in fact, the situation is no different for the sub-Laplacian. The need for weighted norms actually arises from studying the behavior of derivatives near characteristic points of the boundary. Since we believe that identifying the relevant growth rate on the source term is one of the novel points of this paper, we highlight the sharpest result we have in this direction, Theorem \ref{ScaleInvariantHolderEstimateNormalDerivative}, which establishes a second order expansion near the characteristic boundary point of the half-space $\H^n_+=\{t>0\}$. We state this result here in the special case of the sub-Laplacian.

\begin{theorem*}[Theorem \ref{ScaleInvariantHolderEstimateNormalDerivative} for sub-Laplacian]
Suppose $u \in C^2(B_{4}(0)\cap \H^n_+) \cap C(\overline{B_{4}(0)\cap \H^n_+})$ solves
\begin{equation}\label{BVP-for-sub-Laplacian}
\begin{cases}
\Delta_Xu = f & \quad \text{in } B_{4}(0)\cap \H^n_+, \\
u = 0 & \quad \text{on } B_{4}(0)\cap \{t=0\},
\end{cases}
\end{equation}
for some $f \in L^{\infty}(B_{4}(0)\cap \H^n_+,|x|^2)$. Then $\partial_t u(0,0)$ exists. Moreover, there exist constants $C > 1$ and $\rho_0, \a \in(0,1)$ depending only on $n$ such that for all $(x,t) \in B_{\rho_0}(0)\cap \H^n_+$,
\begin{equation}\label{Schauder-estimate-for-sub-Laplacian}
|u(x,t) - \partial_t u(0,0) t| \leq C\left( ||u||_{L^{\infty}(B_{4}(0)\cap \H^n_+)}  + ||f||_{L^{\infty}(B_{4}(0)\cap \H^n_+,|x|^{2})} \right) d^{2+\alpha}((x,t),(0,0)).
\end{equation}
\end{theorem*}

To more concretely illustrate the degree of regularity implied by the estimate \eqref{Schauder-estimate-for-sub-Laplacian}, we recall the expansion of a function $u$ in terms of the gradient and Hessian with respect to the left-invariant vector fields $X$ of $\H^n$ (see \eqref{vector-fields-of-Hn}, \eqref{horizontal-gradient}, \eqref{horizontal-Hessian} below). The (intrinsic) second order Taylor polynomial of a sufficiently smooth function $u$ at $(0,0)$ is given by
$$T_2 u(x,t):=u(0,0)+\langle \nabla_Xu(0,0),x\rangle + \frac{1}{2}\langle D^2_Xu(0,0) x,x \rangle + \partial_t u(0,0) t.$$
We refer to \cite[formula (20.24)]{BLU2007} for a derivation. Since the solution $u$ of \eqref{BVP-for-sub-Laplacian} vanishes on $\{t=0\}$, the only non-trivial term in $T_2 u$ is $\partial_t u(0,0)$. Therefore, \eqref{Schauder-estimate-for-sub-Laplacian} implies $u$ separates from its second order Taylor polynomial at the origin at a rate of $d^{2+\a}$, which can be thought of as a ``punctual'' $C^{2,\a}$ type estimate at the only characteristic point of the domain $\H^n_+$.

Jerison \cite{Jerison81-2} shows that, for the scale-invariant domains $\Gamma_M := \{t>M|x|^2\}$ for $M \in \R$, the validity of higher-order estimates around characteristic boundary points is tied to the value of $M$. Indeed, the regions $\Gamma_M$, despite being smooth domains, behave like 1-homogeneous cones do in the Euclidean setting. Consequently, the growth rate near the origin of a $\Delta_X$-harmonic function in $\Gamma_M$ that vanishes on $\partial \Gamma_M$ is determined by the opening of the parabola $t = M|x|^2$, similar to how the growth rate near the vertex of a harmonic function in a convex cone that vanishes on the boundary of the cone is dependent on the cone angle. For the case $M = 0$ corresponding to the half-space $\H^n_+$, second order Schauder-type estimates near the origin do not necessarily hold for equations with source term $f \in L^{\infty}$, as alluded to in  \cite[Theorems 5.1', 5.2' and discussion at the end of pg. 235]{Jerison81-2}). The significance of Theorem \ref{ScaleInvariantHolderEstimateNormalDerivative} is that it identifies a subclass of $L^{\infty}$ source terms for which solutions of \eqref{BVP-for-sub-Laplacian} enjoy second order estimates. A more detailed discussion of these matters is postponed to Section \ref{sec:halfspace}, where we provide an explicit example of $f\in L^\infty$ for which Theorem \ref{ScaleInvariantHolderEstimateNormalDerivative} fails to hold (see Example \ref{example}), and also provide an application to Dirichlet-type problems (see Corollary \ref{ScaleInvariantHolderEstimateNormalDerivativeWithDirichletData}).

We also point out that the norm of the solution $u$ on the right-hand side of the estimate \eqref{Schauder-estimate-for-sub-Laplacian} is \emph{not} a weighted one. Comparing with the corresponding estimate in the uniformly elliptic case (see, for instance, 
\cite[Theorem 1.2.16]{HanBook}), one expects to see the $L^{\infty}$ norm of the normal derivative on the right hand side of the estimate, which in the setting of $\H^n_+$ corresponds to the weighted norm $||u||_{L^{\infty}(B_{4}(0)\cap \H^n_+, t)}$. Theorem \ref{primadinocond} shows that this weighted norm can be controlled by the usual $L^{\infty}$ norm. Such a ``linear-in-$t$'' estimate is independently interesting, as it shows that in $\H^n_+$, one can do better than the general Lipschitz regularity results obtained in Section \ref{sec:lipschitz} which, due to the anisotropic nature of the metric \eqref{defmetric}, yield an estimate of order $\sqrt{t}$. As illustrated by the second order Taylor polynomial $T_2 u$, linear growth in the $t$ variable corresponds to second order behavior, and so upgrading to an estimate that is of order $t$ is a non-trivial task.

\subsection{Outline of the Paper} Section \ref{sec:preliminaries} conveys some preliminary notions in the Heisenberg group and identifies a family of subsolutions that will be useful for barrier constructions. The important inhomogeneous growth lemma is proved in Section \ref{sec:growthlemma}, where some standard applications to interior regularity are also provided for the reader's convenience. The study of boundary regularity begins in Section \ref{sec:Holder}, where we prove H\"older estimates under appropriate geometric hypotheses on the domain boundary. In Section \ref{sec:lipschitz}, we move on to study boundary regularity of derivatives and introduce the class of weighted $L^{\infty}$ source terms that are necessary to deal with boundary behavior at characteristic points. The final Section \ref{sec:halfspace} showcases the higher order results that can be obtained in the special setting of $\H^n_+$.

\subsection*{Acknowledgment} FA acknowledges support from the National Science Foundation research grant DMS-2246611. GT is partially supported by the Gruppo Nazionale per l’Analisi Matematica, la Probabilità e le loro Applicazioni (GNAMPA) of the Istituto Nazionale di Alta Matematica (INdAM).


\section{Setup and Preliminaries}\label{sec:preliminaries}


Denote points in $\R^{2n+1}$ by $z = (x,t) = (x_1,\ldots, x_{2n}, t) \in \R^{2n} \times \R$, and denote by $\la \cdot, \cdot \ra$ the standard inner product in $\R^{2n}$. Let $\mathbb{I}_n$ denote the $n \times n$ identity matrix, and let $\J$ denote the standard $2n \times 2n$ symplectic matrix, defined in \eqref{standard-symplectic}.

The Heisenberg group $\Hn$ is the homogeneous, stratified Lie group $(\R^{2n+1}, \circ, \delta_r)$ equipped with the (non-commutative) composition law
\[
(x,t) \circ (\xi, \tau) := \left(x + \xi, t + \tau + 2\la \J x, \xi \ra \right),
\]
and the family of anisotropic dilations
\[
\delta_{r} : \Hn \rightarrow \Hn, \ \delta_{r}(x,t) = (r x, r^2 t), \ r>0.
\]
The identity element of $\Hn$ is $0=(0,0)$ and the inverse is $(x,t)^{-1} := (-x,-t)$. The function 
$$
\rho(z) = \rho(x,t) := (|x|^4 + t^2)^{\frac{1}{4}}
$$
defines a $\delta_r$-homogeneous norm on $\Hn$, which induces the metric
\begin{equation}\label{defmetric}
d(z, \zeta) := \rho(z^{-1} \circ \zeta).
\end{equation}
The corresponding metric balls are denoted $B_r(z) := \left\{\zeta \in \R^{2n+1} : d(z, \zeta) < r \right\}$; when $z=0$, we will often simply write $B_r$. We have the equivalent characterizations 
$$B_r(z) =  z \circ B_r = z \circ (\delta_r (B_1)).$$
The Haar measure on $\Hn$ is $(2n+1)$-dimensional Lebesgue measure, which we will denote by $| \cdot |$. As the Jacobian of the map $z \mapsto \delta_r (z)$ is $r^{2n+2}$, we have $|B_r(z)| = |B_r| = r^{2n+2}|B_1|$ for all $z\in\Hn$ and $r>0$; the number $Q := 2n+2$ is called the homogeneous dimension of $\H^n$.

The Lie algebra of $\Hn$ is generated by the horizontal vector fields
\begin{equation}\label{vector-fields-of-Hn}
X_j := \partial_{x_j} + 2(\J x)_j \partial_t, \qquad j = 1,\ldots, 2n.
\end{equation}
The only non-trivial commutation relations among these vector fields are
\[
[X_j, X_{j+n}] = 4 \partial_t, \qquad j = 1, \ldots, n.
\]
Let us note some invariance properties of the vector fields $X_j$. For any $u \in C^1(\R^{2n+1})$, we have the translation invariance property
\[X_i [u(\zeta^{-1}\circ z)] = (X_i u) (\zeta^{-1}\circ z) \quad \text{for all } z, \zeta \in \H^n.\]
Similarly, for any $r > 0$, the rescaled function $u_r(z) := u(\delta_r(z))$ satisfies the dilation property
\[X_j u_r(z) = r X_j u(\delta_r(z)) \quad \text{for all } z \in \H^n.
\]
Denote the horizontal gradient of a function $u \in C^1(\R^{2n+1})$ as
\begin{equation}\label{horizontal-gradient}
\nabla_X u := (X_1 u, \ldots, X_{2n} u) \in \R^{2n},
\end{equation}
and the horizontal Hessian of a function $u \in C^2(\R^{2n+1})$ as
\begin{equation}\label{horizontal-Hessian}
D^2_X u := \left(X_{ij} u \right)_{i,j = 1,\ldots, 2n} \in \R^{2n \times 2n}, \qquad \text{where} \ X_{ij} u := \frac{1}{2}\left(X_i X_j u + X_j X_i u \right).
\end{equation}

\subsection{Non-Divergence Form Operators and Subsolutions}
Given constants $0<\lambda\leq \Lambda$, we denote by $M_n(\lambda,\Lambda)$ the set of symmetric $2n\times 2n$ matrices $M$ satisfying the uniform ellipticity condition
\begin{equation}\label{ellipticity}
\lambda \mathbb{I}_{2n} \leq M \leq \Lambda \mathbb{I}_{2n}.
\end{equation}
Let $O \subset \R^{2n}$ denote a fixed open set. We will be concerned with the second order, non-divergence form operators
\begin{equation}\label{operator-definition}
\mathcal{L}_A := \text{tr}\left(A(z) D^2_X \ \cdot \right) = \sum\limits_{i,j = 1}^{2n} a_{ij}(z) X_{ij} = \sum\limits_{i,j = 1}^{2n} a_{ij}(z) X_i X_j,
\end{equation}
where $A(z) = (a_{ij}(z))_{i,j = 1,\ldots, 2n}$ belongs to $M_n(\lambda,\Lambda)$ for each $z\in O$. When $A(z) \equiv \mathbb{I}_{2n}$, $\L_A$ is the standard sub-Laplacian $\Delta_X = \sum\limits_{i=1}^{2n} X_i^2$. The representation of $\L_A$ in the standard coordinates of $\R^{2n+1}$ is given in \eqref{L-in-coords}.

For open and bounded sets $D$ compactly contained in $O$ and $f : \overline{D} \to \R$ belonging to appropriate function spaces (typically subsets of $L^{\infty}(D)$), we will consider sufficiently smooth solutions to the equation $\Lop_A u= f$ in $D$ which may vanish continuously on appropriate subsets of $\partial D$. Our main goal is to establish \emph{a priori} regularity estimates (both in the interior and at the boundary) which are dependent only on structural constants (i.e. any parameter whose value depends solely on $Q$ and $\frac{\Lambda}{\lambda}$) and suitable norms of the right-hand side $f$. 

Since we will rely on barrier arguments, the following well known result will be indispensable.

\medskip
 
\noindent \textbf{Weak Comparison Principle for $\L_A$.}
\emph{
Let $D$ be an open and bounded set compactly contained in $O$. If  $\Lop_A u \geq \Lop_A v$ in $D$ and $u \leq v$ on $\partial D$, then $u \leq v$ in $\overline{D}$.}
\vskip 0.3cm

We conclude this section with the following lemma, where we collect some differential identities and inequalities that will be needed for the construction of barriers in subsequent sections. Let us define
\begin{equation}\label{defPhiPsi}
\phi(z) := \rho(z)^4 = |x|^4 + t^2,\quad\qquad\quad \psi_{\a}(z) := \phi^{-\a}(z) \quad \text{for } \a \in \R.
\end{equation}

\begin{Lemma}\label{subsolutionlemma} 
For any $M\in M_n(\lambda,\Lambda)$, we have
\begin{equation}\label{subsolutionHgeneric}
\tr \left(M D^2_X \psi_{\a}(z)\right) \geq 0 \qquad  \text{for all } z \in \Hn\setminus\{0\} \,  \text{ and  }\,\,\a\geq \frac{1}{4}\left( (Q+1)\frac{\Lambda}{\lambda} - 3\right).
\end{equation}
\end{Lemma}
\begin{proof} 
We first note the following identities, which hold at any $z=(x,t)\in \Hn$ and follow from direct computation:
\begin{equation}\label{(i)}
X_j \phi(z) = 4x_j |x|^2 + 4t (\J x)_j,
\end{equation}
\begin{equation}\label{(ii)}
X_{ij} \phi(z) = 4\delta_{ji} |x|^2 + 8 x_i x_j + 8 (\J x)_i (\J x)_j.
\end{equation}
Consequently,
$$X_{ij} \psi_\alpha(z) =  \a \phi^{-\alpha-2}(z) \left[(\alpha + 1) X_i \phi(z)X_j \phi(z) - \phi(z)X_{ij} \phi(z) \right]$$
and so, using the identity $| \nabla_X \phi(z)|^2= 16 |x|^2 \phi(z)$, (which is a consequence of \eqref{(i)}), we obtain
$$
X_{ij} \psi_\alpha(z)=\a \phi^{-\alpha-1}(z) \left[16(\alpha + 1)|x|^2 \frac{X_i \phi(z) X_j \phi(z)}{|\nabla_X \phi(z)|^2} - X_{ij} \phi(z) \right].
$$
Hence, for any $M\in M_n(\lambda,\Lambda)$, we have
\begin{align}\label{menouno}
& \tr \left(M D^2_X \psi_{\a}(z)\right) \\
&= \frac{\a}{\phi^{\alpha+1}(z)} \left[16(\alpha + 1)|x|^2 \la M \frac{\nabla_X \phi(z)}{|\nabla_X \phi(z)|}, \frac{\nabla_X \phi(z)}{|\nabla_X \phi(z)|} \ra - \text{tr}\left(M D^2_X \phi(z) \right) \right]\notag \\
&= \frac{4 \a |x|^2}{\phi^{\alpha+1}(z)} \left[4(\alpha + 1) \la M \frac{\nabla_X \phi(z)}{|\nabla_X \phi(z)|}, \frac{\nabla_X \phi(z)}{|\nabla_X \phi(z)|} \ra - \text{tr}(M) - 2 \la M \frac{x}{|x|}, \frac{x}{|x|} \ra - 2 \la M \frac{\J x}{| \J x|}, \frac{\J x}{|\J x|} \ra \right]. \notag
\end{align}
Note that we have used \eqref{(ii)} in the final equality. Since $M \in M_n(\lambda,\Lambda)$, we have for any unit vector $e\in \R^{2n}$ 
$$
\lambda\leq \left\langle Me,e\right\rangle \leq \Lambda \qquad \text{and} \qquad \tr(M) - \left\langle Me,e\right\rangle \leq  (2n-1)\Lambda.
$$ 
Using these inequalities in \eqref{menouno}, we infer
\begin{align}\label{weightss}
\tr \left(M D^2_X \psi_{\a}(z)\right)&\geq  \frac{4 \a |x|^2}{\phi^{\alpha+1}(z)}  \left[(4\alpha + 3)\lambda - (2n+3)\Lambda\right]\notag\\
&=  \frac{4 \lambda \a |x|^2}{\phi^{\alpha+1}(z)} \left( 4\a + 3 - (Q+1)\frac{\Lambda}{\lambda} \right)
\end{align}
which is non-negative if $4\a\geq (Q+1)\frac{\Lambda}{\lambda} - 3$. Note that $(Q+1)\frac{\Lambda}{\lambda} - 3\geq Q-2=2n>0$.
\end{proof}


\section{Inhomogeneous Growth Lemma}\label{sec:growthlemma}


Our goal in this section is to prove the so-called inhomogeneous growth lemma, a fundamental result that has roots in the seminal work of Landis \cite{LandisBook} and Krylov-Safonov \cite{KrylovSafonov80}. We begin by stating a condition on the sub-ellipticity ratio of the operator $\Lop_A$ that will make an appearance in several places.
\begin{Definition} We say the operator $\Lop_A$ satisfies the \emph{Cordes-Landis} condition in $O$ if $A(z) \in M_n(\lambda, \Lambda)$ for each $z \in O$ with
\begin{equation}\label{CordesLandis}\tag{CL}
\frac{\Lambda}{\lambda} < \frac{Q+3}{Q+1}.
\end{equation}
\end{Definition}
Note that \eqref{CordesLandis} is equivalent to
\[\frac{1}{4}\left( (Q+1)\frac{\Lambda}{\lambda} - 3\right) < \frac{Q}{4}.\]
We can thus fix a constant $\a$ such that
\begin{equation}\label{CordesLandisa}
\frac{1}{4}\left( (Q+1)\frac{\Lambda}{\lambda} - 3\right)  \leq \a  < \frac{Q}{4}.
\end{equation}
This guarantees the function $\psi_{\a}$ defined in \eqref{defPhiPsi} satisfies the subsolution property \eqref{subsolutionHgeneric} uniformly among the class of operators $\Lop_A$ with $A(z)\in M_n(\lambda,\Lambda)$ for each $z\in O$ and satisfying \eqref{CordesLandis}.

Since $\a <\frac{Q}{4}$, $\psi_{\a}$ is locally integrable around the origin and so, for any $E \subset \Hn$ bounded and measurable, the following function is well defined:
\begin{equation}\label{definitionofUE}
U_E(z) := \int_E  \psi_{\a}(\zeta^{-1}\circ z) \ d\zeta = \int_E  d(z,\zeta)^{-4\a} \ d\zeta.
\end{equation}
This function will be key to the barrier arguments leading to the proof of the growth lemma. The following lemma highlights some of its useful properties.

\begin{Lemma}\label{subsolutionUbounded}
Suppose $\L_A$ satisfies \eqref{CordesLandis} in $O$. Let $\a$ be as in \eqref{CordesLandisa} and $E \subset \Hn$ be a bounded and measurable set. Then
\begin{equation}\label{subsolutionH}
\L_A U_E(z) \geq 0 \quad\mbox{ for all }z\in O\setminus  \overline{E}.
\end{equation}
Moreover, there exist structural constants $\t\geq 4$ and $C_1, C_2, C_3 > 0$ satisfying $2C_1 \leq C_3<2C_2$ such that, if $E \subset B_r(z_0)$ for some $r>0$ and $z_0\in\Hn$, we have the following bounds for $U_E$:
\begin{itemize}
\item[(i)] $U_E(z) \leq C_1 r^{-4\a}|E| \ \text{ for all } z \in \partial B_{\t r}(z_0).$
\item[(ii)] $U_E(z) \leq C_2 r^{-4\a} |B_r(z_0)| \ \text{ for all } z \in B_{\t r}(z_0).$
\item[(iii)] $U_E(z) \geq C_3 r^{-4\a}|E| \ \text{ for all } z \in B_r(z_0).$
\end{itemize}
\end{Lemma}
\begin{proof} 
Fix $A(z)\in M_n(\lambda,\Lambda)$ for any $z\in O$. By \eqref{subsolutionHgeneric} and the left-invariance of the vector fields $X_i$, we have
$$\tr \left(A(z) D^2_X \psi_{\a}(\zeta^{-1}\circ z)\right) \geq 0 \qquad  \text{for all }\zeta\in \Hn \text{ and }  z \in O \setminus\{\zeta\}.$$
The desired inequality \eqref{subsolutionH} follows once we recognize that, for $z\notin\overline{E}$, we can differentiate twice under the integral sign.

We proceed to establish the stated bounds for $U_E$; the relations satisfied by the constants $\t, C_1, C_2, C_3$ will become evident by the end of the proof. 
\begin{itemize}
\item[(i)] Let $z \in \partial B_{\t r}(z_0)$ with $\t>1$. Then $d(z,\zeta) \geq (\t - 1)r$ for all $\zeta \in E$ and so, keeping in mind that $\a>0$,
$$U_E(z) \leq ((\t - 1)r)^{-4\a}|E| = C_1 r^{-4\a}|E| \qquad \text{for all } z \in \partial B_{\t r}(z_0).$$
\item[(ii)] Let $z \in B_{\t r}(z_0)$. Then $E \subset B_{(\t + 1) r}(z)$ and so, keeping in mind the group-translation invariance of  Lebesgue measure,
\begin{align*}
U_E(z) & \leq \int_{B_{(\t +1) r}(z)} d(z,\zeta)^{-4\a} \ d\zeta \\
& = ((\t + 1)r)^{Q - 4\a} \int_{B_1} \rho(\zeta)^{-4\a} \ d\zeta \\
& = \sigma (\t + 1)^{Q - 4\a} |B_r(z_0)| r^{-4\a}, \qquad \text{where } \sigma := \frac{1}{|B_1|} \int_{B_1} \rho(\zeta)^{-4\a} \ d \zeta.
\end{align*}
We stress that we have used, in a crucial way, the property $4\a<Q$, for otherwise $\sigma$ would not be finite. Therefore,
$$U_E(z) \leq C_2 r^{-4\a}|B_r(z_0)| \qquad \text{for all } z \in B_{\t r}(z_0).$$
\item[(iii)] Let $z \in B_r(z_0)$. Then $d(z,\zeta) \leq 2r$ for all $\zeta \in E$ and so
$$U_E(z) \geq (2r)^{-4\a}|E| = C_3r^{-4\a}|E| \qquad \text{for all } z \in B_r(z_0).$$
\end{itemize}
Finally, we choose $\t > 1$ to satisfy both 
$$(\t - 1)^{-4\a} \leq 2^{-4\a-1}\quad\mbox{ and }\quad\sigma (\t + 1)^{Q - 4\a} > 2^{-4\alpha-1}.$$ 
Note that it is possible to choose $\tau$ sufficiently large thanks to the fact that $4\a<Q$ (i.e. there exists $\t_0>3$ so that any $\t\geq \t_0$ is a feasible choice). This guarantees both $C_1=(\t - 1)^{-4\a} \leq 2^{-4\a-1}= \frac{C_3}{2}$ and $C_2=\sigma (\t + 1)^{Q - 4\a}>2^{-4\a-1}= \frac{C_3}{2}$ as desired. Observe that these constants depend only on $Q$ and $\frac{\Lambda}{\lambda}$.
\end{proof}

We are now ready to prove a Landis-type growth lemma for subsolutions of $\L_A u = f$ under the assumption \eqref{CordesLandis}. From here onward, the constant $\t\geq 4$ will denote the one determined in Lemma \ref{subsolutionUbounded}. We also employ the convention $f^+=\max\{f,0\}$ and $f^-=\max\{-f,0\}$, so that $f^+,f^-\geq0$ and $f=f^+-f^-$.

\begin{Theorem}[Growth Lemma]\label{growthlemmalandisH} Suppose $B_{\t r}(z_0) \Subset O$ and $\L_A$ satisfies \eqref{CordesLandis} in $O$. Let $D \subset B_{\t r}(z_0)$ be such that $D \cap B_{r}(z_0) \neq \emptyset$. Suppose $u \in C^2(D) \cap C(\overline{D})$ is non-negative in $D$, vanishes on $\partial D \cap B_{\t r}(z_0)$ and satisfies $\L_A u \geq f$ in $D$ for some $f \in L^{\infty}(D)$. Then there exists a structural constant $\eta \in (0,1)$ such that
$$\sup_{D \cap B_r(z_0)} u\leq \left(1 - \eta \frac{|B_r(z_0) \setminus D|}{|B_r(z_0)|} \right) \sup_D u +
\frac{1}{4n\lambda} ||f^-||_{L^{\infty}(D)} (\tau r)^2.$$
\end{Theorem}

\begin{proof} Let $C_1, C_2, C_3 > 0$ be the constants from Lemma \ref{subsolutionUbounded}. Let $E:= B_r(z_0) \setminus D$ with $z_0=(x_0,t_0)$ and $r>0$, and let $U_E$ be as in \eqref{definitionofUE}. 

First, consider the function
$$w(z) := \left(\sup_D u \right) \left[1 - \frac{r^{4\a}}{C_2 |B_r(z_0)|} \left(U_E(z) - C_1 r^{-4\a}|E| \right) \right].$$
We know from \eqref{subsolutionH} that 
\[\Lop_A U_E(z) \geq 0 \quad \text{ for all } z \in B_{\t r}(z_0) \setminus \overline{E},\]
and so $\L_A w \leq 0$ on $D$. By properties (i) and (ii) in Lemma \ref{subsolutionUbounded}, we have, respectively, that 
\[w \geq \underset{D}{\sup} \ u \geq u \quad \text{ on } \partial B_{\t r}(z_0) \cap D\]
and $w \geq 0$ on $B_{\t r}(z_0)$. Since $u = 0$ on $\partial D \cap B_{\t r}(z_0)$, we conclude that $w \geq u$ on $\partial D \cap B_{\t r}(z_0)$. Therefore, $w \geq u$ on $\partial D$.

Next, consider the function
$$v(z) := w(z) + \frac{F}{4n\lambda}\left[(\tau r)^2 - |x - x_0|^2 \right],$$
where $F := ||f^-||_{L^{\infty}(D)}$. Note that $v \geq w$ on $B_{\t r}(z_0)$, as $|x - x_0| \leq \t r$ if $z = (x,t) \in B_{\t r}(z_0)$. Since $w \geq u$ on $\partial D$, it follows that $v \geq u$ on $\partial D$. 

We claim $\L_A v \leq \L_A u$ on $D$. Indeed, if $\varphi(z) = \varphi(x) = |x-x_0|^2$, then $D^2_X \varphi(z) = D^2\varphi(x) = 2 \mathbb{I}_{2n}$. It follows that
\begin{equation}\label{4nla}
\L_A \varphi(z) = \tr(A(z) D^2_X \varphi(z)) = \tr(A(z) D^2 \varphi(x)) = 2 \tr(A(z)) \geq 4n \lambda.
\end{equation}
Since $\L_A w \leq 0$ on $D$, we conclude that 
\[\L_A v \leq -F \leq -f^-\leq f \leq \L_A u\quad \text{on } D.\]
where we have used the hypothesis $\L_A u \geq f$ in the final inequality. The comparison principle yields $v \geq u$ on $D$. In particular, $v \geq u$ on $D \cap B_r(z_0)$. 

We now prove an upper bound on $v$ in $D \cap B_r(z_0)$. By Lemma \ref{subsolutionUbounded} property (iii), we have for all $z \in B_r(z_0)$
\begin{align*}
v(z) & \leq \left(\sup_D u \right) \left[1 - \frac{r^{4\a}}{C_2 |B_r(z_0)|} \left(C_3 r^{-4\a}|E| - C_1 r^{-4\a}|E| \right) \right] + \frac{F}{4n\lambda} (\tau r)^2 \\
& = \left(\sup_D u \right) \left[1 - \frac{|E|}{C_2 |B_r(z_0)|} \left(C_3 - C_1 \right) \right] + \frac{F}{4n\lambda} (\tau r)^2\\
& \leq \left(\sup_D u \right) \left[1 - \frac{|E|}{|B_r(z_0)|} \left(\frac{C_3}{2C_2} \right) \right] + \frac{F}{4n\lambda} (\tau r)^2.
\end{align*}
Note that, in the final inequality, we have used $C_1 \leq \frac{1}{2}C_3$. Setting 
$$\eta := \dfrac{C_3}{2C_2} \in (0,1)$$ 
and using the previously established fact $u \leq v$ on $D \cap B_r(z_0)$, we conclude
$$
\sup_{D \cap B_r(z_0)} u \leq \left(1 - \eta \frac{|E|}{|B_r(z_0)|} \right) \sup_D u + \frac{F}{4n\lambda} (\tau r)^2.
$$
This is the desired inequality once we substitute the definitions of $E$ and $F$.
\end{proof}

\begin{Remark}
The arguments presented thus far can be carried out for coefficient matrices $A(z)$ satisfying the slightly less restrictive condition
$$
\sup_{z\in O}\left\{\frac{{\rm{tr}}(A(z)) + 4 \max\limits_{|v|=1}\left\langle A(z)v,v \right\rangle }{\min\limits_{|v|=1}\left\langle A(z)v,v \right\rangle}\right\}<Q+4
$$
(cf. \cite[condition (1.2)]{T2014}). This is similar to (and inspired by) the condition that appears in Landis' work \cite{LandisBook, LandisDokl}.
\end{Remark}


\subsection{Applications to interior regularity}

We make a quick digression to illustrate some applications of Theorem \ref{growthlemmalandisH} to interior regularity. Specifically, we will prove interior H\"older estimates for solutions of $\L_A u =f$ and a scale-invariant inhomogeneous Harnack inequality for non-negative solutions, assuming throughout that $\L_A$ satisfies \eqref{CordesLandis} and $f \in L^{\infty}$. The ideas involved in the proofs of these interior estimates are well known to experts; we only present them here for the convenience of the reader.

Let us state the definitions of $d$-H\"older continuity and, for future reference, $d$-Lipschitz continuity.

\begin{Definition}
A function $u$ defined on a set $\O \subset \H^n$ is said to be locally $d$-H\"older continuous of order $\beta \in (0,1)$ at $z_0 \in \overline{\O}$ if there exist constants $C, r_0 >0$ such that
$$|u(z) - u(z_0)| \leq C \ d(z, z_0)^{\beta} \qquad \text{ for all } \ z \in B_{r_0}(z_0) \cap \overline{\O}.$$
If we can let $\beta = 1$ above, then we say $u$ is locally $d$-Lipschitz continuous.
\end{Definition}

\begin{Remark}
It follows from the definition of the metric $d$ in \eqref{defmetric} that $d$-H\"older continuity implies H\"older continuity with the respect to the the standard metric in $\R^{2n+1}$, possibly with a different constant $C' > 0$ and exponent $\beta' \in (0,1)$ (cf. \cite[Proposition 5.1.6]{BLU2007}). Note, however, that $d$-Lipschitz continuity does not necessarily imply Lipschitz continuity with the respect to the standard metric in $\R^{2n+1}$.
\end{Remark}

We begin with the proof of local H\"older continuity of solutions to $\L_A u = f$, Corollary \ref{oscillationdecay}, which is a direct consequence of Theorem \ref{growthlemmalandisH}. To the best of our knowledge, the only results comparable to Corollary \ref{oscillationdecay} are the ones in \cite{GT2011, T2014}. Both these results assume a restriction on the ellipticity ratio, but they are stated for solutions of the homogeneous equation $\L_A u = 0$, and are obtained as a consequence of scale-invariant Harnack inequalities. We also refer the reader to our previous work with Guti\'errez \cite{AGT2017}, where we substituted the restriction on the ellipticity ratio with a control on the modulus of continuity of the matrix coefficients. 

In what follows we use the notation
$$
\underset{D}{\osc} \ u := \sup_{D} u - \inf_{D} u.
$$

\begin{Corollary}\label{oscillationdecay} Suppose $B_{\t r}(z_0) \Subset O$ and $\L_A$ satisfies \eqref{CordesLandis} in $O$. Suppose $u \in C^2(B_{\t r}(z_0)) \cap C(\overline{B_{\t r}(z_0)})$ solves $\L_A u = f$ in $B_{\t r}(z_0)$ for some $f\in L^{\infty}(B_{\t r}(z_0))$. Then there exists a structural constant $\mu \in (0, 1)$ such that
$$\underset{B_r(z_0)}{\osc} \ u \leq \mu \underset{B_{\t r}(z_0)}{\osc} \ u + \frac{1}{3n\lambda}  ||f||_{L^{\infty}(B_{\t r}(z_0))} (\tau r)^2.$$
Consequently, $u$ is locally $d$-H\"older continuous of any order $\beta<\min\left\{\frac{\log(\mu^{-1})}{\log(\tau)},1\right\}$, and for some constant $C$ depending on $\beta$ and on the $L^\infty$-norms of $f$ and $u$ in $B_{\tau r}(z_0)$. 
\end{Corollary}

\begin{proof}
This is a standard argument; we provide some details for the reader's convenience.

Consider the function
$$v(z) := u(z) - \frac{1}{2}\left(\sup_{B_r(z_0)} u + \inf_{B_r(z_0)} u \right), \qquad z \in B_{\t r}(z_0).$$
Let $D^+ = \left\{v > 0 \right\} \cap B_{\t r}(z_0)$ and $D^- = \left\{v < 0 \right\} \cap B_{\t r}(z_0)$. Since $|B_r(z_0) \setminus D^+|+|B_r(z_0) \setminus D^-|\geq |B_r(z_0)|$, one of the two inequalities $|B_r(z_0) \setminus D^+| \geq \frac{1}{2}|B_r(z_0)|$ and $|B_r(z_0) \setminus D^-| \geq \frac{1}{2}|B_r(z_0)|$ must hold. With no loss of generality, we may assume that $|B_r(z_0) \setminus D^+| \geq \frac{1}{2}|B_r(z_0)|$; otherwise, replace $v$ in the following argument with $-v$. 

Since $\L_A v = f$, we can use Theorem \ref{growthlemmalandisH} to obtain
\begin{align*}
\sup_{B_r(z_0)} v= \sup_{D^+ \cap B_r(z_0)} v &\leq \left(1-\frac{\eta}{2}\right)\sup_{D^+} v + \frac{1}{4n\lambda} ||f^-||_{L^{\infty}(B_{\t r}(z_0))} (\tau r)^2 \\
&\leq \left(1-\frac{\eta}{2}\right)\sup_{B_{\t r}(z_0)} v + \frac{1}{4n\lambda} ||f||_{L^{\infty}(B_{\t r}(z_0))} (\tau r)^2
\end{align*}

In the previous estimate, we have also assumed that $D^+\cap B_r(z_0)\neq\emptyset$, for otherwise $v\leq 0$ in $B_r(z_0)$ and the estimate holds trivially. Since, by definition, we have
$$
\sup_{B_r(z_0)} v  = \frac{1}{2} \underset{B_r(z_0)}{\osc} \ u\quad\mbox{ and }\quad
\sup\limits_{B_{\t r}(z_0)} v\leq \underset{B_{\t r}(z_0)}{\osc} \ u - \frac{1}{2} \underset{B_r(z_0)}{\osc} \ u,
$$
we deduce
$$
\left(1-\frac{\eta}{4}\right) \underset{B_r(z_0)}{\osc} \ u \leq  \left(1-\frac{\eta}{2}\right) 
\underset{B_{\t r}(z_0)}{\osc} \ u  + \frac{1}{4n\lambda} ||f||_{L^{\infty}(B_{\t r}(z_0))} (\tau r)^2.
$$
This implies the oscillation decay in the statement with the choice
$$\mu := \frac{1 - \frac{\eta}{2}}{1 - \frac{\eta}{4}} \in (0,1).$$ 
The local $d$-H\"older continuity of $u$ now follows by applying, for instance, \cite[Lemma 8.23]{GTBook}. \end{proof}

Our next application of Theorem \ref{growthlemmalandisH} is to the proof of an inhomogeneous Harnack inequality for non-negative solutions of $\L_A u = f$ with $f \in L^{\infty}$. This result will be used later in Section \ref{subsec-linear-growth} when we study higher regularity of solutions in a characteristic half-space.

While it is possible to prove the Harnack inequality using the growth lemma directly \cite{LandisBook}, we elect to use the axiomatic approach developed in \cite{DiFGutLan, Safometric, GM2018} for brevity. Specifically, we will apply Theorems 2.7 and 2.8 from \cite{GM2018} and so the following proof will entail verifying that we can invoke these results. We also note that a Harnack inequality for the homogeneous equation $\L_A u = 0$, such as the ones proved in \cite{GT2011, T2014, AGT2017} implies an inhomogeneous Harnack inequality when the right-hand-side $f$ belongs to $L^{\infty}$; see, for instance, \cite[proof of Theorem 5.5]{Gutierrez-Lanconelli-2003}. Such an argument relies on the linearity of the operator $\L_A$, whereas the axiomatic approach can potentially be applied to nonlinear problems as well.

\begin{Theorem}[Inhomogeneous Harnack Inequality]\label{thm:Harnack}
Suppose $\L_A$ satisfies \eqref{CordesLandis} in $O$. There exist structural constants $C_H, K_H > 1$ such that if $B_{K_H R}(z)\Subset O$ and $u \in C^2(B_{K_H R}(z)) \cap C(\overline{B_{K_H R}(z)})$ is a non-negative solution of  $\Lop_A u = f$ in $B_{K_H R}(z)$ with $f \in L^{\infty}(B_{K_H R}(z))$, then
\begin{equation}\label{HarnacK}
\sup_{B_R(z)} u \leq C_H\left(\inf_{B_R(z)} u + R^2 ||f||_{L^{\infty}(B_{K_H R}(z))} \right).
\end{equation}
\end{Theorem}
\begin{proof} Consider any open set $\Omega$ with closure contained in $O$. We first show how the growth lemma, Theorem \ref{growthlemmalandisH}, implies the following $\epsilon$-critical density property for any fixed $\epsilon \in \left(0, 1 \right)$: for every $u \in C^2(B_{\t r}(z_0)) \cap C(\overline{B_{\t r}(z_0)})$ with $B_{\t r}(z_0)\subset \Omega$ we have
\begin{equation}\label{epscritical}
\begin{cases}
u\geq 0 \quad\quad\,\,\mbox{ in }B_{\t r}(z_0)\\
\L_A u \leq f  \quad\mbox{ in }B_{\t r}(z_0)\\
|\{u \geq 1 \} \cap B_{ r}(z_0)| \geq \epsilon |B_{ r}(z_0)|\\
r^2||f^+||_{L^{\infty}(B_{\t r}(z_0))} \leq \frac{2n\lambda}{\tau^2} \eta \epsilon
\end{cases}
\quad
\Longrightarrow
\quad
\inf_{B_r(z_0)} u \geq \frac{ \eta \epsilon}{2}.
\end{equation}
Consider any $u$ as in the left-hand side of \eqref{epscritical}, and let $v := 1 - u$ and $D := \left\{v > 0 \right\} \cap B_{\t r}(z_0)$. If $D\cap B_{ r}(z_0)=\emptyset$, then $u\geq 1$ in $B_{ r}(z_0)$ which trivially implies $\inf_{B_r(z_0)} u\geq 1$. We can then assume $D\cap B_{ r}(z_0)\neq \emptyset$. We notice that we have $\L_A v \geq -f$ and $v \leq 1$ on $B_{\t r}(z_0)$, and $v$ vanishes on $\partial D\cap B_{\t r}(z_0)$. Moreover,
$$
\frac{|B_r(z_0) \setminus D|}{|B_r(z_0)|}=\frac{|\left\{v \leq 0 \right\} \cap B_{r}(z_0)|}{|B_r(z_0)|}
 = \frac{|\{u \geq 1 \} \cap B_{r}(z_0)|}{|B_r(z_0)|} \geq \epsilon.
$$
Applying Theorem \ref{growthlemmalandisH} to $v$ in $B_{\t r}(z_0)$, we thus obtain
\begin{align*}
\left(1 - \inf_{B_r(z_0)} u \right)&= \sup_{B_r(z_0)} v = \sup_{D \cap B_r(z_0)} v\\
&\leq \left(1 - \eta \frac{|B_r(z_0) \setminus D|}{|B_r(z_0)|} \right) \sup_D v +
\frac{1}{4n\lambda} ||(-f)^-||_{L^{\infty}(D)} (\tau r)^2 \\
&\leq \left(1 - \eta \epsilon \right) \sup_D v +
\frac{1}{4n\lambda} ||f^+||_{L^{\infty}(D)} (\tau r)^2 \\
& \leq \left(1 - \eta \epsilon \right) +
\frac{1}{4n\lambda} ||f^+||_{L^{\infty}(D)} (\tau r)^2\\
& \leq 1 - \eta \epsilon  +
\frac{\tau^2}{4n\lambda} \frac{2n\lambda}{\tau^2} \eta \epsilon= 1 - \eta \epsilon + \frac{1}{2} \eta \epsilon=1-\frac{1}{2} \eta \epsilon.
\end{align*}
Rearranging terms, we obtain 
$$\inf_{B_r(z_0)} u  \geq \frac{ \eta \epsilon}{2},$$
which finishes the proof of \eqref{epscritical}.

Since \eqref{epscritical} holds true for any $\epsilon\in (0,1)$ we can also verify the following double ball property: for every $u \in C^2(B_{\t r}(z_0)) \cap C(\overline{B_{\t r}(z_0)})$ with $B_{\t r}(z_0)\subset \Omega$ we have
\begin{equation}\label{doubleball}
\begin{cases}
u\geq 0 \quad\quad\,\,\mbox{ in }B_{\t r}(z_0)\\
\L_A u \leq f  \quad\mbox{ in }B_{\t r}(z_0)\\
\inf_{B_\frac{r}{2}(z_0)} u \geq 1\\
r^2||f^+||_{L^{\infty}(B_{\t r}(z_0))} \leq \frac{2n\lambda}{2^{Q}\tau^2} \eta 
\end{cases}
\quad
\Longrightarrow
\quad
\inf_{B_r(z_0)} u \geq \frac{ \eta }{2^{Q+1}}.
\end{equation}
As a matter of fact, if we consider any $u$ as in the left-hand side of \eqref{doubleball} and we assume by contradiction that $\inf_{B_r(z_0)} u < \frac{ \eta }{2^{Q+1}}$, then we can apply \eqref{epscritical} with $\epsilon=2^{-Q}$ which yields $|\{u \geq 1 \} \cap B_{ r}(z_0)| < 2^{-Q} |B_{ r}(z_0)|$. Noticing that $\{u \geq 1 \}$ contains $B_\frac{r}{2}(z_0)$, we have already obtained the following contradiction
$$
\frac{1}{2^Q}=\frac{|B_\frac{r}{2}(z_0)|}{|B_{r}(z_0)|}\leq \frac{|\{u \geq 1 \} \cap B_{ r}(z_0)|}{|B_{ r}(z_0)|}<\frac{1}{2^{Q}}
$$
which proves \eqref{doubleball}.

Combining \eqref{epscritical} 
and \eqref{doubleball} with the results in \cite{DiFGutLan, GM2018} we can now deduce the desired Harnack inequality. More precisely, for $f\in L^\infty(\Omega)$ we define $S_{\Omega}(B_r(z_0), f) := r^2 ||f||_{L^{\infty}(\Omega)}$ and we let $\mathbb{K}_{\Omega,f}$ to be the set of non-negative $C^2$-functions $u$ defined in (a domain containing a) subset of $\Omega$ and satisfying there $\L_A u = f$. From \eqref{epscritical}, \eqref{doubleball}, and \cite[Theorems 2.7 and 2.8]{GM2018}, we infer the existence of constants $C_H, K_H>1$ (depending only on the constants $\tau, \eta, Q, \lambda$ that appear in \eqref{epscritical}-\eqref{doubleball}) such that, for any ball $B_{K_H R}(z)\subseteq \Omega$ and any non-negative $C^2$-solution $u$ of $\L_A u = f$ in $B_{K_H R}(z)$, we have
$$\sup_{B_R(z)} u \leq C_H\left(\inf_{B_R(z)} u + R^2 ||f||_{L^{\infty}(\Omega)} \right).$$
Setting $\Omega=B_{K_H R}(z)$, we have thus established \eqref{HarnacK}.
\end{proof}


\section{Boundary H\"older Regularity}\label{sec:Holder}

We now embark on the task of establishing boundary regularity results for solutions of $\Lop_A u = f$ near a portion of the boundary where $u$ vanishes. This will follow from boundary versions of the growth lemma, Theorem \ref{growthlemmalandisH}, and oscillation decay, Corollary \ref{oscillationdecay}. Such estimates will naturally depend on the boundary geometry of the domain $\Omega$ where the equation is satisfied. We will see that, under suitable regularity assumptions on $\partial\Omega$, we can directly apply Theorem \ref{growthlemmalandisH} if we assume the Cordes-Landis condition \eqref{CordesLandis}. On the other hand, we will also show that, under stronger regularity hypotheses on $\partial\Omega$, we can dispense with the condition \eqref{CordesLandis} and prove oscillation decay close to the boundary points where $u$ vanishes. This is noteworthy, as there are no interior regularity results available in the literature in this regime.  

We begin by precisely stating the necessary regularity hypotheses on $\partial\Omega$. These are analogues of the well known ``exterior-cone'' condition in Euclidean space, which appears frequently in the literature concerning boundary H\"older regularity for uniformly elliptic equations \cite{Miller67, 
Michael81, ChoSafonov}.

\begin{Definition}
We say that $\O \subset \H^n$ satisfies the \emph{positive exterior density condition} at $z_0\in\partial\O$ if there exist $\theta_0 \in (0,1]$ and $r _0>0$ such that
$$
|B_r(z_0) \setminus \O|  \geq \theta_0 |B_r(z_0)|\quad\mbox{ for all }0<r\leq r_0.
$$
We say that $\O$ satisfies the \emph{uniform positive exterior density condition} if $\O$ satisfies the positive exterior density condition at every boundary point $z_0\in\partial\O$ for $\theta_0, r_0$ that can be chosen uniformly with respect to $z_0\in\partial\O$.
\end{Definition}

\begin{Definition}\label{contball}
We say that $\O \subset \H^n$ satisfies the \emph{exterior ball containment condition} at $z_0\in\partial\O$ if there exist $\theta\in (0,1]$ and $r_0>0$ such that
\begin{equation}\label{pointneeded}
\mbox{ for all } 0<r\leq r_0, \,\,\mbox{ there exists }z_r\in \Hn\mbox\, \mbox{ with }\, B_{\theta r}(z_r)\subseteq B_r(z_0) \setminus \O.
\end{equation}
We say that $\O$ satisfies the \emph{uniform exterior ball containment condition} if $\O$ satisfies the exterior ball containment condition at every boundary point $z_0\in\partial\O$ for $\theta_0, r_0$ that can be chosen uniformly with respect to $z_0\in\partial\O$.
\end{Definition}
Using the fact that $|B_{\theta r}(z_r)|=\theta^Q |B_r(z_0)|$, it is clear that the exterior ball containment condition implies the positive exterior density condition with $\theta_0=\theta^Q$.

There are several cone-type sets whose boundaries satisfy the regularity properties defined above; see, for instance, the general construction described in \cite[Theorem 6.5]{LU2010}. Since we are working in a homogeneous Lie group, it is possible to use the dilation $\delta_r$ to explicitly define cone-like sets with vertex at a chosen point. We outline this construction below and, consequently, establish the exterior ball containment condition for these sets.

\begin{Example}\label{exa:exterior-cone}
We say a non-empty open set $\mathcal{C}_0\subset \Hn$ is \emph{a truncated open cone with vertex at $0$} if
$$
\delta_r(z)\in \mathcal{C}_0 \text{ for all } z\in
\mathcal{C}_0 \text{ and } r\in (0,1).$$
We say $\mathcal{C}\subset \Hn$ is a \emph{truncated open cone with vertex at $z_0$} if $\mathcal{C}=z_0\circ \mathcal{C}_0$ for some truncated open cone $\mathcal{C}_0$ with vertex at $0$.

Suppose there exists a truncated open cone $\mathcal{C}$ with vertex at $z_0 \in \partial \O$ such that $\mathcal{C}\subset \Hn\setminus \O$. Then $\mathcal{C}_0=z_0^{-1}\circ \mathcal{C}$ is a truncated open cone with vertex at $0$. Since $\mathcal{C}_0$ is open, we can find $\bar{z}\in \mathcal{C}_0$ and $\bar{\theta}>0$ such that $B_{\bar{\theta}}(\bar{z})\subset \mathcal{C}_0$. Using the definition of $\mathcal{C}_0$ and the degree one $\delta_r$-homogeneity of the metric $d$, we have 
\[B_{\bar{\theta} r}(\delta_r(\bar{z}))=\delta_r\left( B_{\bar{\theta}}(\bar{z}) \right)\subset \mathcal{C}_0 \quad \text{for all } 0<r\leq 1.\] 
On the other hand, since $d(\delta_{r}(\bar{z}),0)=r\rho(\bar{z})$, the triangle inequality implies \[B_{\bar{\theta}r}(\delta_r(\bar{z}))\subset \mathcal{C}_0\cap B_{r(\bar{\theta}+\rho(\bar{z}))}(0) \quad \text{for all } 0<r\leq 1.
\]
If we make the replacement $r(\bar{\theta}+\rho(\bar{z}))\mapsto r$, we find that
\[
B_{\theta r}\left(\delta_{\frac{r}{\bar{\theta}+\rho(\bar{z})}}(\bar{z})\right)\subset \mathcal{C}_0\cap B_{r}(0)\quad\mbox{ for any }0<r\leq \bar{\theta}+\rho(\bar{z}), \quad\mbox{with the choice }\theta=\frac{\bar{\theta}}{\bar{\theta}+\rho(\bar{z})}.
\]
If we also fix a positive $r_0<\min\{\bar{\theta}+\rho(\bar{z}), R_0\}$, then for all $0<r\leq r_0$, we have
\[
B_{\theta r}(z_r)\subset\mathcal{C}\cap B_{r}(z_0)\subset B_{r}(z_0)\setminus \O \quad \text{where } z_r:=z_0\circ \delta_{\frac{r}{\bar{\theta}+\rho(\bar{z})}}(\bar{z}).
\]
This shows $\O$ satisfies the exterior ball containment condition at $z_0 \in \partial \O$.
\end{Example}

We are now ready to prove a boundary version of Theorem \ref{growthlemmalandisH} under appropriate hypotheses on $\Lop_A$ and the geometry of the boundary $\partial \O$.  Our approach is inspired by that of Cho and Safonov in the uniformly elliptic case \cite{ChoSafonov}. Note that the hypothesis (H1) below does not impose any upper bound on the subellipticity ratio $\frac{\Lambda}{\lambda}$.

\begin{Theorem}[Boundary Growth Lemma]\label{bogrth}
Let $\L_A$ be such that $A(z)\in M_n(\lambda,\Lambda)$ for any $z\in O$.
Consider an open set $\O \Subset O$ and let $z_0\in\partial\O$. Assume either
\begin{itemize}
\item[(H1)] $\O$ satisfies the exterior ball containment condition at $z_0$, or
\item[(H2)]  $\O$ satisfies the exterior density condition at $z_0$ and $\Lop_A$ satisfies \eqref{CordesLandis}.
\end{itemize}
Let $D$ be an open set such that $D\subseteq \O\cap B_{\t r}(z_0)$ and $D\cap B_r(z_0)\neq\emptyset$ for some $0<r\leq r_0$. Suppose $u \in C^2(D) \cap C(\overline{D})$ is non-negative in $D$, vanishes on $\partial D \cap B_{\t r}(z_0)$, and satisfies $\L_A u \geq f$ in $D$ for some $f \in L^{\infty}(D)$.  Then there exists a constant $\gamma \in (0,1)$ (depending only on structural constants and on $\theta$ or, respectively, on $\theta_0$) such that
\begin{equation}\label{bogrthestimate}
    \sup_{D \cap B_r(z_0)} u \leq \gamma \, \sup_{D} u + \frac{1}{4n\lambda} ||f^-||_{L^{\infty}(D)} (\tau r)^2.
\end{equation}
\end{Theorem}
\begin{proof} The proof under the assumption (H2) is a direct application of Theorem  \ref{growthlemmalandisH}. Indeed, since $\O$ satisfies the positive exterior density condition at $z_0$ and $D\subseteq \O$, we have
$$
|B_r(z_0) \setminus D|\geq |B_r(z_0) \setminus \O|\geq \theta_0 |B_r(z_0)|.
$$
Hence, by Theorem  \ref{growthlemmalandisH}, we conclude that
\begin{align*}
\sup_{D \cap B_r(z_0)} u\leq \left(1 - \eta \theta_0 \right)\sup_D u + \frac{1}{4n\lambda}||f^-||_{L^{\infty}(D)} (\tau r)^2,
\end{align*}
which proves the statement with the choice $\gamma:=1 - \eta \theta_0$.

We are thus left with the proof under the assumption {\rm{(H1)}}. In this case, let $\theta$ and $z_r=(x_r,t_r)\in B_r(z_0) \setminus \O$ be as in \eqref{pointneeded}. If we let $D':=D\cap B_{3r}(z_r)$, then since $D\subseteq \O$, we have
\begin{equation}\label{pointneededhere}
 B_{\theta r}(z_r)\subseteq B_r(z_0) \setminus \O\subseteq B_r(z_0) \setminus D \subseteq B_r(z_0) \setminus D'.
\end{equation}
Let $F = ||f^-||_{L^{\infty}(D')}$. Fix $\a\geq \frac{1}{4}\left( (Q+1)\frac{\Lambda}{\lambda} - 3\right)$ and let $\psi_{\a}$ be as in \eqref{defPhiPsi}; we remind the reader that the Cordes-Landis condition \eqref{CordesLandis} is \emph{not} in effect, since we are assuming (H1). Consider the function
$$v(z) := w(z) + \frac{F}{4n\lambda}\left[(3r)^2 - |x - x_r|^2 \right],$$
where
$$
w(z):=\left(\sup_D u \right) \frac{(\theta r)^{-4\a}-\psi_\a(z_r^{-1}\circ z)}{(\theta r)^{-4\a}-(3 r)^{-4\a}}.
$$
We notice that the choice of $\a$, together with the left-invariance of the vector fields $X_i$ and the fact that $z_r\notin D'$, allows us to use \eqref{subsolutionHgeneric} and infer that $\L_A( \psi_\a(z_r^{-1}\circ \cdot))(z)\geq 0$ in $D'$. Since $\a>0$ and $\theta\leq 1$, we have $(\theta r)^{-4\a}-(3 r)^{-4\a} > 0$ and so $\L_A w\leq 0$ in $D'$. Combining this with \eqref{4nla}, we conclude that
\begin{equation}\label{perdopo}
\L_A v \leq -F \leq -f^-\leq  f\leq \L_A u \quad\mbox{ in }D'.
\end{equation}
Let us next compare the functions $v$ and $u$ on $\partial D'$. Since $\overline{D'}\subseteq \overline{B_{3 r}(z_r)}$, we have $|x-x_r|\leq d(z,z_r)\leq 3r$ for all $z\in\partial D'$ which implies $v\geq w$ on $\partial D'$. On the other hand, for any $z\in\partial B_{3 r}(z_r) \cap \overline{D'}$ we have $w(z)=\sup_D u \geq u(z)$. Furthermore, since $B_{\theta r}(z_r)\cap D'=\emptyset$ by \eqref{pointneededhere} and $B_{3 r}(z_r)\subseteq B_{4 r}(z_0)\subseteq B_{\t r}(z_0)$, we also have $w(z)\geq 0=u(z)$ for any $z\in \partial D'\cap B_{3 r}(z_r)$. We have thus shown $v\geq w\geq u$ on $\partial D'$. It follows from \eqref{perdopo} and the comparison principle that $v\geq u$ in $D'$. In particular, since $D'\supseteq D\cap B_{2r}(z_r)\supseteq D\cap B_{r}(z_0)\neq \emptyset$, we deduce
\begin{align*}
\sup_{D \cap B_r(z_0)} u\leq \sup_{D \cap B_r(z_0)} v\leq \sup_{D \cap B_{2r}(z_r)} v &\leq \left(\sup_D u \right) \frac{(\theta r)^{-4\a}-(2 r)^{-4\a}}{(\theta r)^{-4\a}-(3 r)^{-4\a}}+\frac{F}{4n\lambda}(3r)^2\\
&=\left(1-\frac{2^{-4\a}-3^{-4\a}}{\theta^{-4\a}-3^{-4\a}}\right) \sup_D u+\frac{F}{4n\lambda}(3r)^2.
\end{align*}
Recalling that $F\leq ||f^-||_{L^{\infty}(D)}$, this establishes the desired estimate with the choice $\gamma:= 1-\frac{2^{-4\a}-3^{-4\a}}{\theta^{-4\a}-3^{-4\a}}$.  
\end{proof}

\begin{Remark}
We note the fact that, in the case of assumption {\rm{(H1)}}, any $\t\geq 4$ can be used in proof above (i.e. $\t$ is not required to be the constant chosen after Lemma \ref{subsolutionUbounded}). 
\end{Remark}

In the previous theorem, we have some freedom in choosing the domain $D$ close to the boundary point $z_0\in\partial\O$ where we can apply the growth estimate \eqref{bogrthestimate}. By choosing $D=\O\cap B_{\t r}(z_0)$ we immediately obtain oscillation decay and $d$-H\"older continuity at $z_0$ for non-negative subsolutions vanishing on a portion of $\partial\O$ containing $z_0$. This is summarized in the following corollary, in which the constant $\gamma\in(0,1)$ is the one from \eqref{bogrthestimate}. 

\begin{Corollary}
Consider an open set $\O \Subset O$, and let $z_0\in\partial\O$. Assume either {\rm{(H1)}} or {\rm{(H2)}} from Theorem \ref{bogrth} holds. Suppose there exists $r_1\in (0,r_0]$ such that $u \geq 0$ and $\L_A u \geq f$ in $\O\cap B_{\t r_1}(z_0)$ for $f \in L^{\infty}(\O\cap B_{\t r_1}(z_0))$, and $u = 0$ on $B_{\t r_1}(z_0) \cap \partial \O$. Then $u$ is $d$-H\"older continuous at $z_0$; that is, for any $0<\beta<\min\left\{\frac{\log(\gamma^{-1})}{\log(\tau)},1\right\}$ there exists a constant $C > 0$ (depending on $\beta$ and on the $L^\infty$ norms of $f$ and $u$) such that
$$u(z) \leq C \ d(z, z_0)^{\beta} \qquad \text{ for all } \ z \in B_{r_1}(z_0) \cap \O.$$
\end{Corollary}
\begin{proof}
For every $r \in (0,r_1)$ we can apply Theorem \ref{bogrth} to the function $u$ in the region $D = \O\cap B_{\tau r}(z_0)$. Notice that $D\cap B_r(z_0)=\O\cap B_{ r}(z_0)\neq\emptyset$ since $z_0\in\partial\O$. Moreover, since $u$ is non-negative in $D$ and vanishes at $z_0$, we have $\inf_{B_{\rho}(z_0) \cap \O} u = 0$ for any $\rho \in (0,\tau r_1)$. Therefore, by \eqref{bogrthestimate},
\begin{align*}
\underset{B_{r}(z_0) \cap \O}{\osc}  u & =\sup_{B_r(z_0) \cap D} u \\
& \leq \gamma \sup_{D} u + \frac{1}{4n\lambda} ||f^-||_{L^{\infty}(D)} (\tau r)^2 \\
& =  \gamma \underset{B_{\t r}(z_0) \cap \O}{\osc}  u + \frac{1}{4n\lambda} ||f^-||_{L^{\infty}(D)} (\tau r)^2.
\end{align*}
A standard iterative argument (as in the proof of Corollary \ref{oscillationdecay}) yields the desired H\"older estimate at $z_0$.
\end{proof}

Using Theorem \ref{bogrth}, we can also derive uniform boundary H\"older estimates for (possibly sign-changing) solutions to $\L_A u =f$ in $\Omega$ that vanish on $\partial\O$. This is the content of the following result where, once again, the constant $\gamma\in(0,1)$ is the one from \eqref{bogrthestimate}. 

\begin{Theorem}\label{thm:bdy-Holder}
Consider an open set $\O\Subset O\cap B_{R_0}(0)$. Assume either
\begin{itemize}
\item[(uH1)] $\O$ satisfies the uniform exterior ball containment condition, or
\item[(uH2)] $\O$ satisfies the uniform positive exterior density condition and $\Lop_A$ satisfies \eqref{CordesLandis}.
\end{itemize} 
Suppose $u\in C^2(\O)\cap C(\overline{\O})$ solves the inhomogeneous Dirichlet problem
$$
\begin{cases}
\L_A u = f & \quad \text{in } \O \\
u = 0 & \quad \text{on } \partial\O
\end{cases}
$$
for some $f \in L^{\infty}(\O)$. Then for any $0<\beta<\min\left\{\frac{\log(\gamma^{-1})}{\log(\tau)},1\right\}$, there exists a constant $C_\beta > 0$ such that
\begin{equation}\label{bdryHolderestimate}
|u(z)| \leq C_\beta\, ||f||_{L^{\infty}(\O)} \ {\rm{dist}}(z, \partial\O)^{\beta} \qquad \text{ for all } \ z \in \O.
\end{equation}
In particular, $u$ is $\beta$-H\"older continuous at $\partial \O$.
\end{Theorem}
\begin{proof}
The proof follows the strategy from \cite[Theorem 3.5]{ChoSafonov}. A straightforward application of the maximum principle to the functions $v_{\pm}=\pm u-\frac{||f||_{L^{\infty}}}{4n\lambda}\left(R^2_0-|x|^2\right)$ (see also \eqref{4nla} above) yields
\begin{equation}\label{inftoinf}
||u||_{L^{\infty}(\O)}\leq \frac{R_0^2}{4n\lambda}||f||_{L^{\infty}(\O)}\qquad\mbox{for all }z\in\O.
\end{equation}
Denote $\O^+=\O\cap\{u>0\}$ and $\O^-=\O\cap\{u<0\}$. Without loss of generality, we may assume that $\O^+$ and $\O^-$ are both non-empty. We will show that, for any $0<\beta<\min\left\{\frac{\log(\gamma^{-1})}{\log(\tau)},1\right\}$, there exists a constant $C_\beta > 0$ such that
\begin{equation}\label{halfthisisok}
\sup_{\O^+} \left(\frac{u}{{\rm{dist}}(\cdot, \partial\O)^{\beta}}\right) \leq C_\beta\, ||f||_{L^{\infty}}\quad\mbox{ and }\quad \sup_{\O^-} \left(\frac{-u}{{\rm{dist}}(\cdot, \partial\O)^{\beta}}\right) \leq C_\beta\, ||f||_{L^{\infty}}.
\end{equation} 
Once the estimates in \eqref{halfthisisok} are verified, we will immediately obtain \eqref{bdryHolderestimate}.

It suffices to prove only the first estimate in \eqref{halfthisisok} as the second one will follow from the first by considering $-u$ instead of $u$. Consider $c>0$ small enough so that the set $\O_c=\O\cap\{u>c\}$ is non-empty. Since $u$ is continuous and vanishes on $\partial\O$, we have ${\rm{dist}}(\O_c, \partial\O)^{\beta}>0$ and the function $ z\mapsto {\rm{dist}}(z, \partial\O)^{-\beta} \left(u(z) - c\right)$ belongs to $C\left(\overline{\O_c}\right)$. Hence, there exists $\zeta_0\in \O_c$ such that
$$
{\rm{dist}}(\zeta_0, \partial\O)^{-\beta} \left(u(\zeta_0) - c\right)=\sup_{\O_c}\frac{u-c}{{\rm{dist}}(\cdot, \partial\O)^{\beta}} =:M>0.
$$
Even though $M$ depends on $c$, we will suppress this dependence, as our goal is to obtain an upper bound on $M$ independent of $c$. Once we prove such a bound, we can let $c \to 0^+$ and obtain \eqref{halfthisisok}.

Denote $r={\rm{dist}}(\zeta_0, \partial\O)>0$ and let $z_0\in\partial\O$ be such that $d(\zeta_0,z_0)=r$. If $r\geq r_0$, then by \eqref{inftoinf} we readily obtain
\begin{equation}\label{Mfacile}
M\leq \frac{u(\zeta_0)}{r^\beta}\leq \frac{R_0^2}{4n\lambda r_0^\beta }||f||_{L^{\infty}}.
\end{equation}
On the other hand, if $0<r\leq r_0$, consider the function
$$
\tilde{u}=u-c-(\epsilon r)^\beta M
$$
for $\epsilon\in (0,1)$ to be determined. Notice that $\tilde{u}(\zeta_0)=Mr^\beta(1-\epsilon^\beta)>0$. Consider the open set
$$
D=\O_c \cap B_{\t r}(z_0)\cap \{\tilde{u}>0\}. 
$$
We already know that $\zeta_0\in D\cap \partial B_{r}(z_0)$ which implies $D\cap B_r(z_0)\neq \emptyset$. We also notice that $D\subseteq \{{\rm{dist}}(\cdot, \partial\O)>\epsilon r\}$ since 
$$(\epsilon r)^\beta M<u(z)-c\leq {\rm{dist}}(z, \partial\O)^\beta M \qquad \text{ for any } z\in D.$$ 
Furthermore, since $\tilde{u}=-(\epsilon r)^\beta M<0$ on $\partial \O_c$, we have $\tilde{u}=0$ on $\partial D\cap  B_{\t r}(z_0)$. We can thus apply Theorem \ref{bogrth} to obtain 
$$
\tilde{u}(\zeta_0)\leq\sup_{D \cap B_r(z_0)} \tilde{u} \leq \gamma \, \sup_{D} \tilde{u} + \frac{1}{4n\lambda} ||f||_{L^{\infty}} (\tau r)^2,
$$
where the first inequality is a consequence of the continuity of $u$. Keeping in mind that $z_0\in\partial\O$ we have for any $z\in \O_c\cap B_{\t r}(z_0)$ 
$$u(z)-c\leq M {\rm{dist}}(z, \partial\O)^\beta \leq M (\t r)^\beta$$
which allows us to deduce
\begin{align*}
M r^\beta=u(\zeta_0)-c &=\tilde{u}(\zeta_0)+Mr^\beta \epsilon^\beta \\
& \leq \gamma \, M r^\beta \left( \t ^\beta -\epsilon^\beta\right) + \frac{1}{4n\lambda} ||f||_{L^{\infty}} (\tau r)^2 + Mr^\beta \epsilon^\beta.
\end{align*}
Rearranging, we have
$$
M\leq M (\gamma \tau^\beta + (1-\gamma)\epsilon^\beta )+ \frac{\t^2}{4n\lambda} ||f||_{L^{\infty}} r^{2-\beta}.
$$
Since $\beta<\frac{\log(\gamma^{-1})}{\log(\tau)}$ implies $\gamma\t^\beta <1$, we can choose $\epsilon\in (0,1)$ (depending only on $\t, \gamma, \beta$) such that $\gamma \tau^\beta + (1-\gamma)\epsilon^\beta < 1$. With such a choice, we finally obtain the bound
\begin{equation}\label{Mdifficile}
M\leq  \frac{\t^2 r_0^{2-\beta}}{1-\gamma \tau^\beta - (1-\gamma)\epsilon^\beta}\frac{||f||_{L^{\infty}}}{4n\lambda}
\end{equation}
in the case $0<r\leq r_0$. Combining \eqref{Mfacile} and \eqref{Mdifficile}, we find that
$$
M\leq C_\beta ||f||_{L^{\infty}}\quad\mbox{ where }C_\beta=\frac{1}{4n\lambda}\max\left\{\frac{\t^2 r_0^{2-\beta}}{1-\gamma \tau^\beta - (1-\gamma)\epsilon^\beta},\frac{R_0^2}{ r_0^\beta }\right\}.
$$
Therefore, we have for all $c > 0$ sufficiently small
$$
\frac{u(z)-c}{{\rm{dist}}(z, \partial\O)^{\beta}}\leq C_\beta ||f||_{L^{\infty}}\qquad\mbox{ for all }z\in\O_c.
$$
Since the constant $C_{\beta}$ is independent of $c$, we can let $c\to 0^+$ and obtain
$$
\frac{u(z)}{{\rm{dist}}(z, \partial\O)^{\beta}}\leq C_\beta ||f||_{L^{\infty}}\qquad\mbox{ for all }z\in\O^+,
$$
which completes the proof of \eqref{halfthisisok}.
\end{proof}


\section{Boundary Lipschitz Regularity}\label{sec:lipschitz}


In this section, we initiate the study of derivative estimates for the inhomogeneous equation $\L_A u = f$ on portions of the boundary where the solution $u$ vanishes. Compared to the results of the previous section, the results proved in this section will require stronger hypotheses not only on the domain geometry, but also on the function $f$. Indeed, it turns out that $f \in L^{\infty}$ will not suffice and more precise assumptions must be made about the behavior of $f$ near the so-called characteristic points of the domain boundary. We note, however, that our results are new even for the homogeneous equation $\L_A u = 0$, and that we do not impose the hypothesis \eqref{CordesLandis} in this section. A similar approach for homogeneous equations in non-divergence form has been carried out in \cite{MartinoTralli2016} where analogues of the classical Hopf lemma are established under appropriate geometric conditions at characteristic boundary points.

Let us begin by stating the regularity hypothesis on the boundary that we will need.

\begin{Definition}\label{extball}
We say that $\O$ satisfies the \emph{exterior touching ball condition} at $z_0\in\partial\O$ if there exist $r_0>0$ and $p_0\in\Hn$ such that 
\begin{equation}\label{touching}
B_{r_0}(p_0) \subset \Hn \setminus \overline{\O} \quad\mbox{ and }\quad z_0\in \partial B_{r_0}(p_0).
\end{equation}
\end{Definition}

By \cite[Remark 3.5]{Lanconelli-Uguzzoni-PoissonKernel}, every (Euclidean) convex subset of $\H^n$ satisfies the exterior ball condition at any boundary point (and for arbitrary $r_0$). It is also proved in \cite[Lemma 3.2]{Lanconelli-Uguzzoni-PoissonKernel} that, under an exterior ball condition at the point $z_0$ and provided the boundary data $\varphi$ vanish in a neighborhood of $z_0$, any solution of $\Delta_X u=0$ with boundary value $\varphi$ satisfies a $d$-Lipschitz bound in a neighborhood of $z_0$. We extend this result to the case of general non-divergence form operators $\L_A$ with a source term $f$. 

In the barrier argument carried out below, $f$ will need to belong to a (possibly) smaller space than $L^\infty$. Let us fix the formal definition for this functional framework.

\begin{Definition}\label{weightDef}
Given an open bounded set $D \subset \mathbb{H}^n$ and a non-negative function $\omega \in L^{\infty}(D)$, we define the space $L^{\infty}(D,\omega)$ to be the set of all functions $f\in L^{\infty}(D)$ such that 
\[|f(z)| \leq C\omega(z) \quad \text{for all } z \in D.\]
We denote
\[
||f||_{L^{\infty}(D,\omega)} := \inf\{C : |f(z)| \leq C\omega(z) \text{ for all } z \in D \}.
\]
\end{Definition}

The usual $L^{\infty}$ space corresponds to the weight $\omega(z) \equiv 1$. In general, $L^{\infty}(D,\omega) \subset L^{\infty}(D)$, with strict containment if $\omega$ is allowed to vanish on a subset of $\overline{D}$.

The weights $\omega$ appearing in the next result will be expressed in terms of the horizontal gradient of the distance function $d(z, z_0)$ from a point $z_0 \in \partial \O$. Recall that, from \eqref{(i)}, one has
$$
|\nabla_X\rho(z)|^2=\frac{|x|^2}{\rho^2(z)}\leq 1 \qquad\mbox{ for every }z\in \Hn\setminus\{0\}.
$$

\begin{Theorem}\label{thm:bdy-Lip}
Consider an open set $\O \Subset O$, and let $z_0\in\partial\O$. Assume that $\O$ satisfies the exterior ball condition at $z_0$ as in Definition \ref{extball}, 
with radius $r_0>0$ and center $p_0\in\Hn$. Let $\omega_{p_0}(z)=|\nabla_X\rho (p_0^{-1}\circ z)|^{2}$. Suppose that $u\in C^2(\O \cap B_{3r_0}(z_0))\cap C(\overline{\O \cap B_{3r_0}(z_0)})$ satisfies $\L_A u \geq f$ in $\O \cap B_{3r_0}(z_0)$ with $f \in L^{\infty}(\O\cap B_{3r_0}(z_0),\omega_{p_0})$ and $u \leq 0$ on $B_{3r_0}(z_0) \cap \partial \O$. Then there exist a universal constant $b>1$ and a constant $C>0$ (which depends on universal parameters and, in addition, on $||f^-||_{L^{\infty}(\O\cap B_{3r_0}(z_0),\omega_{p_0})}$ and $||u^+||_{L^{\infty}(\O\cap B_{3r_0}(z_0))}$) such that
$$u(z) \leq C\, d(z, z_0) \qquad \text{ for all } \ z \in \O\cap B_{\frac{r_0}{b}}(z_0).$$
\end{Theorem}
\begin{proof}
Denote $D=\O\cap B_{2r_0}(p_0)$ and $p_0=(p_0^h,p_0^v) \in \R^{2n} \times \R$. By Definition \ref{extball}, we have that $\O\cap B_{3r_0}(z_0) \supseteq D\neq \emptyset$ and $p_0\notin \overline{D}$. Fix $\a, M$ as follows
$$
\a>\frac{1}{4}\left( (Q+1)\frac{\Lambda}{\lambda} - 3\right)\qquad
M:=\max\left\{ ||u^+||_{L^{\infty}(D)}, 
\frac{(2^{4\a}-1)r_0^2||f^-||_{L^{\infty}(D,\omega_{p_0})}}{\lambda \a(4\a + 3 - (Q+1)\frac{\Lambda}{\lambda})}\right\}.
$$
We recall that $\a>0$, and we may also assume $M>0$ (otherwise $u\leq 0$ in $D$ and there is nothing to prove). Define
\begin{equation}\label{barrLip}
w(z)=M \frac{r_0^{-4\a}-\psi_\a(p_0^{-1}\circ z)}{r_0^{-4\a}-(2 r_0)^{-4\a}}.
\end{equation}
Thanks to the choice of $\a$ and the left-invariance of the vector fields $X_i$, we can repeat the computation for establishing \eqref{weightss} to obtain
$$
\L_A \left( \psi_{\a}(p_0^{-1}\circ \cdot)\right)(z)\geq  4\lambda  \a |x-p_0^h|^2 \phi^{-\alpha-1}(p_0^{-1}\circ z) \left( 4\a + 3 - (Q+1)\frac{\Lambda}{\lambda} \right)
$$
for all $z\neq p_0$. Since $\frac{M}{r_0^{-4\a}-(2 r_0)^{-4\a}} \geq 0$, for every $z\in D$ we deduce

\begin{align*}
\L_A w (z) &\leq \frac{-4\lambda \a M}{r_0^{-4\a}-(2 r_0)^{-4\a}} \frac{ 4\a + 3 - (Q+1)\frac{\Lambda}{\lambda}}{\rho^{4\alpha+2}(p_0^{-1}\circ z)}\frac{|x-p_0^h|^2}{\rho^2(p_0^{-1}\circ z)}\\
&\leq \frac{-4\lambda \a M}{r_0^{-4\a}-(2 r_0)^{-4\a}} \frac{ 4\a + 3 - (Q+1)\frac{\Lambda}{\lambda}}{(2 r_0)^{4\a+2}}\frac{|x-p_0^h|^2}{\rho^2(p_0^{-1}\circ z)}\\
&=-M|\nabla_X\rho (p_0^{-1}\circ z)|^{2}\frac{\lambda \a(4\a + 3 - (Q+1)\frac{\Lambda}{\lambda})}{(2^{4\a}-1)r_0^2}.
\end{align*}
The choices of $M$ and $\omega_{p_0}$ yield
\begin{align*}
\L_A w (z) &\leq -\omega_{p_0}(z) ||f^-||_{L^{\infty}(D,\omega_{p_0})} 
\leq -f^-(z)\leq f(z) \quad \text{for all } z \in D.
\end{align*}
Here, we have used that $f^{-}(z)\leq  ||f^-||_{L^{\infty}(D,\omega_{p_0})}\omega_{p_0}(z)$, which is a direct consequence of Definition \ref{weightDef}. Therefore, $\L_A w\leq \L_A u$ in $D$. 

On the other hand, for any $z\in \overline{\O}\cap \partial B_{2r_0}(p_0)$ we have $u(z)\leq u^+(z)\leq ||u^+||_{L^{\infty}(D)} \leq M =w(z)$. Moreover, for any $z\in B_{2r_0}(p_0)\cap \partial \O $ we have $u(z)\leq 0\leq w(z)$ by the geometric condition \eqref{touching}. Combining these two observations, we find that $u\leq w$ on $\partial D$. The comparison principle then implies $u \leq w$ in $D$. In particular
\begin{equation}\label{nowLagrangia}
u(z)\leq w(z)=w(z)-w(z_0) \quad\mbox{ for all }z\in \O\cap B_{r_0}(z_0).
\end{equation}

We now recall (see, for instance, \cite[Theorem 20.3.1]{BLU2007}) that there exist universal constants $c_1, b_1\geq 1$ such that
$$|g(z)-g(z_0)|\leq c_1 d(z,z_0) \sup_{d(\zeta,z_0)\leq b_1 d(z,z_0)}|\nabla_X g(\zeta)|$$
for any smooth function $g$. Hence, keeping in mind that $w$ is smooth away from $p_0$ and $d(z_0,p_0)=r_0$, we can fix $b=2b_1$ so that $b_1d(z,z_0)\leq \frac{1}{2}r_0$ for any $z\in B_{\frac{r_0}{b}}(z_0)$ and notice that
\begin{align*}
\sup_{\zeta\in B_{\frac{r_0}{2}}(z_0)}|\nabla_X w(\zeta)|&\leq \frac{4\a M }{r_0^{-4\a}-(2 r_0)^{-4\a}}\sup_{\zeta\in B_{\frac{r_0}{2}}(z_0)}\frac{|\nabla_X \rho(p_0^{-1}\circ\zeta)|}{\rho^{4\a+1}(p_0^{-1}\circ\zeta)} \\
& \leq \frac{4\a M }{r_0^{-4\a}-(2 r_0)^{-4\a}} \frac{2^{4\a+1}}{r_0^{4\a+1}}\\
& = \frac{\a 2^{4\a+3} }{1-2^{-4\a}} \frac{M}{r_0}
\end{align*}
Setting the final expression above to be $C$, we thus have
$$|w(z)-w(z_0)|\leq C d(z,z_0)\quad \mbox{ for all }z\in B_{\frac{r_0}{b}}(z_0).$$
The proof is complete once we insert the previous bound into \eqref{nowLagrangia}. 
\end{proof}

\begin{Corollary}\label{cor5}
With the same setup as Theorem \ref{thm:bdy-Lip}, suppose now that $u\in C^2(\O \cap B_{3r_0}(z_0))\cap C(\overline{\O \cap B_{3r_0}(z_0)})$ solves $\L_A u= f$ in $\O \cap B_{3r_0}(z_0)$ for $f \in L^{\infty}(\O\cap B_{3r_0}(z_0),\omega_{p_0})$, and $u = 0$ on $B_{3r_0}(z_0) \cap \partial \O$. Then $u$ is $d$-Lipschitz continuous at $z_0$. More specifically, for the same constants $b>1$ and $C>0$ as in Theorem \ref{thm:bdy-Lip}, we have
$$|u(z)| \leq C\, d(z, z_0) \qquad \text{ for all } \ z \in \O\cap B_{\frac{r_0}{b}}(z_0).$$
\end{Corollary}

\begin{proof}
It suffices to apply Theorem \ref{thm:bdy-Lip} to both $u$ and $-u$.
\end{proof}

Let us note that the source term $f$ in Theorem \ref{thm:bdy-Lip} and Corollary \ref{cor5} is assumed to belong in a weighted $L^{\infty}$ space for a weight that depends on the point $p_0$ appearing in Definition \ref{extball}. When $\partial \O$ is $C^1$-smooth, the vector $\nabla_X\rho (p_0^{-1}\circ z_0)$ is the projection on to the horizontal distribution of the normal vector to $\partial \O$ at $z_0$. If $\nabla_X\rho (p_0^{-1}\circ z_0) = 0$, then the point $z_0$ is said to be \emph{characteristic} for $\partial \O$. In this case, the weighted space $L^{\infty}(\O\cap B_{3r_0}(z_0),\omega_{p_0})$ is strictly smaller than $L^{\infty}(\O\cap B_{3r_0}(z_0))$. On the other hand, if $\nabla_X\rho (p_0^{-1}\circ z_0)\neq 0$ and $r_0$ is small enough, the weighted space coincides with the usual $L^\infty$ space. 

In general, the estimate provided by Corollary \ref{cor5} is a ``punctual'' one, as it depends on the weighted norm $||f||_{L^{\infty}(\O\cap B_{3r_0}(z_0),\omega_{p_0})}$. If, instead, we wish to obtain a uniform $d$-Lipschitz estimate in some region $\Gamma \subset \partial \O$, it is not clear if any non-trivial function $f$ will satisfy the requirements imposed by Theorem \ref{thm:bdy-Lip} at every point on $\Gamma$. This illustrates how the boundary regularity theory for degenerate equations at characteristic points can differ drastically from that of uniformly elliptic equations.

\section{Improved Regularity Results in a Characteristic Half-Space}\label{sec:halfspace}

The results of the previous section left the issue of obtaining uniform derivative estimates on a portion of the boundary in questionable status. Specifically, it was unclear what weighted space the inhomogeneous term $f$ should belong to. In this final section, we will resolve this issue in the special case of the half-space
$$\H^n_+:=\{(x,t) \in \H^n : t > 0 \}.$$ 
Note that $\H^n_+$ satisfies the exterior touching ball condition (Definition \ref{extball}), has only a single characteristic boundary point at the origin, and is invariant with respect to intrinsic dilations, making it a model domain for studying boundary derivative estimates near characteristic points. Moreover, any other hyperplane in $\H^n$ passing through the origin is non-characteristic in a neighborhood of the origin. 

We will study the problem
\begin{equation}\label{half-space-BVP}
\begin{cases}
\Lop_A u = f & \quad \text{in } B_r(0)\cap \H^n_+, \\
u= 0 & \quad \text{on }  B_r(0)\cap\{t=0\}.
\end{cases}
\end{equation}
Our goal is to prove estimates for $\partial_t u(0,0)$ depending only on supremum-type bounds of $u$ and $f$. The main results are a growth estimate of order $t$ near the origin, Theorem \ref{primadinocond}, and a second order asymptotic expansion, Theorem \ref{ScaleInvariantHolderEstimateNormalDerivative}. These results are new even in the case $A(\cdot)=\mathbb{I}_{2n}$, i.e. for the sub-Laplacian $\Lop_A=\Delta_X$.

\subsection{Linear-in-$t$ growth near the boundary}\label{subsec-linear-growth}
In this subsection, we begin to sharpen the analysis concerning the Lipschitz estimates of Section \ref{sec:lipschitz} and establish $L^\infty$-bounds for the ratio $\frac{u}{t}$. 

First, notice that since $\Omega=\H^n_+$ satisfies the exterior ball condition at $z_0=(0,0)$ for every $r_0>0$ with $p_0=(0,-r_0^2)$, we may invoke the estimates in Section \ref{sec:lipschitz} with the weight
\[\omega_{p_0}(z)=\frac{|x|^2}{\sqrt{|x|^4+(t+r_0^2)^2}}\leq \frac{|x|^2}{r_0^2} \quad \text{for all } z\in \H^n_+.
\]
This motivates the use of the weighted space $L^{\infty}(B_{r_0}(0)\cap\H^n_+,|x|^2)$ in this section. 

Next, notice that Theorem \ref{thm:bdy-Lip} yields an estimate of the form $u(x,t)\leq C \left(|x|^4+t^2\right)^{\frac{1}{4}}$ which is sublinear in the variable $t$. In the following lemma, we prove an improvement of this estimate.

\begin{Lemma}\label{lemprimogiro}
Fix $r_0>0$. Suppose $u\in C^2(B_{4r_0}(0)\cap\H^n_+)\cap C(\overline{B_{4r_0}(0)\cap\H^n_+})$ satisfies $\L_A u \geq f$ in $B_{4r_0}(0)\cap\H^n_+$ for $f \in L^{\infty}(B_{4r_0}(0)\cap\H^n_+,|x|^2)$, and $u \leq 0$ on $B_{4r_0}(0) \cap \{t=0\}$. Then
$$u(z) \leq \frac{M_0}{r_0^4} \left(|x|^4+r_0^2 t \right) \quad \text{ for all } \ z \in B_{\frac{4}{3}r_0}(0)\cap \H^n_+$$
with
$$M_0=C_0\max\left\{ ||u^+||_{L^{\infty}(B_{4r_0}(0)\cap\H^n_+)}, 
\frac{r_0^4}{\lambda}||f^-||_{L^{\infty}(B_{4r_0}(0)\cap\H^n_+,|x|^2)}\right\}$$
and $C_0$ a structural constant depending only on $Q, \frac{\Lambda}{\lambda}$.
\end{Lemma}
\begin{proof}
We follow the lines of the proof of  Theorem \ref{thm:bdy-Lip} by means of a one-point barrier function. Fix $r=\frac{4}{3}r_0$, and keep in mind that the ball centered at $p_0=(0,-r^2)$ with radius $r$ is an exterior ball for $\H^n_+$ which touches the plane $\{t=0\}$ at $0$. Denoting $D=B_{2r}(p_0)\cap \H^n_+$, we can fix
$$
\a=\frac{1}{4}\left( (Q+1)\frac{\Lambda}{\lambda} - 2\right),
$$
and we define as in \eqref{barrLip} the function
$$
w(z)=M \frac{r^{-4\a}-\psi_\a(p_0^{-1}\circ z)}{r^{-4\a}-(2 r)^{-4\a}}
$$
with 
$$
M=\max\left\{ ||u^+||_{L^{\infty}(B_{4r_0}(0)\cap\H^n_+)}, 
\frac{(2^{4\a}-1)4r^4||f^-||_{L^{\infty}(B_{4r_0}(0)\cap\H^n_+,|x|^2)}}{4\a\lambda}\right\}
$$
As in Theorem \ref{thm:bdy-Lip}, for any $z=(x,t)\in B_{3r}(0)$ with $t>0$ we obtain 
\begin{align*}
\L_A w (z) &\leq \frac{-4\lambda \a M}{r^{-4\a}-(2 r)^{-4\a}} \frac{ 4\a + 3 - (Q+1)\frac{\Lambda}{\lambda}}{(2 r)^{4\a+2}}\frac{|x|^2}{\sqrt{|x|^4+(t+r^2)^2}}\\
&=\frac{-\left( (Q+1)\frac{\Lambda}{\lambda} - 2\right)\lambda M}{2^{\left( (Q+1)\frac{\Lambda}{\lambda} - 2\right)}-1} \frac{|x|^2}{4r^2\sqrt{|x|^4+(t+r^2)^2}}\\
&\leq \frac{-\left( (Q+1)\frac{\Lambda}{\lambda} - 2\right)\lambda M}{2^{\left( (Q+1)\frac{\Lambda}{\lambda} - 2\right)}-1} \frac{|x|^2}{4r^4} \\
& \leq - ||f^-||_{L^{\infty}(B_{4r_0}(0)\cap\H^n_+,|x|^2)} |x|^2\\
&\leq -f^-(z)\leq f(z).
\end{align*}
Hence, $\L_A w \leq \L_A u$ in $D$. Since $w=M\geq u$ on $\partial B_{2r}(p_0)\cap \H^n_+$ and $w\geq 0\geq u$ on $\{t=0\}$, the comparison principle implies $u\leq w$ in $D$. In particular,
$$
u(x,t)\leq M \frac{1-\left(1+\frac{|x|^4+t^2+2tr^2}{r^4}\right)^{-\alpha}}{1-2^{-4\a}}\qquad\mbox{ for any }(x,t)\in B_r(0)\cap \H^n_+.
$$
We can then exploit the convexity of the function $\sigma\mapsto (1+\sigma)^{-\alpha}$ (which implies $1-(1+\sigma)^{-\alpha}\leq \alpha\sigma$ for any $\sigma>-1$) in order to infer that
\begin{align*}
u(x,t)&\leq \frac{\alpha M}{1-2^{-4\a}}\frac{|x|^4+t^2+2tr^2}{r^4} \\
&\leq \frac{\alpha M}{1-2^{-4\a}}\frac{|x|^4+t\left(3r^2\right)}{r^4} \\
&\leq \frac{4\alpha M}{1-2^{-4\a}} \frac{27}{64}\frac{|x|^4+r_0^2 t}{r_0^4}
\end{align*}
for all $(x,t)\in B_r(0)\cap \H^n_+$. This finishes the proof if we denote
\[M_0 = \frac{4\alpha M}{1-2^{-4\a}} \frac{27}{64} \quad \text{and} \quad C_0=\max\left\{ \frac{4\alpha}{1-2^{-4\a}} \frac{27}{64}, 
\frac{2^{4\a}}{3}\right\}.\]
\end{proof}

The previous lemma yields, in particular, a ``linear-in-$t$'' estimate in the subregion $\{|x|^4\leq c t\}$. In order to obtain such an estimate in the full neighborhood $B_{r_0}(0)\cap \H^n_+$ of the boundary point, we are going to exploit the fact that $(0,0)$ is an isolated characteristic point for $\H^n_+$. In the following theorem, we prove the desired estimate by constructing barrier functions at every boundary point.

\begin{Theorem}\label{primadinocond}
Fix $r_0>0$. Suppose $u\in C^2(B_{4r_0}(0)\cap\H^n_+)\cap C(\overline{B_{4r_0}(0)\cap\H^n_+})$ satisfies $\L_A u \geq f$ in $B_{4r_0}(0)\cap\H^n_+$ for $f \in L^{\infty}(B_{4r_0}(0)\cap\H^n_+,|x|^2)$, and $u\leq 0$ on $B_{4r_0}(0) \cap \{t=0\}$. Then we have
\begin{equation}\label{linear-growth-estimate}
u(z)\leq \frac{M_1}{r^2_0} t \qquad \text{ for all } \ z=(x,t) \in B_{r_0}(0)\cap \H^n_+
\end{equation}
with
$$M_1=C_1\max\left\{ ||u^+||_{L^{\infty}(B_{4r_0}(0)\cap\H^n_+)}, 
\frac{r_0^4}{\lambda}||f^-||_{L^{\infty}(B_{4r_0}(0)\cap\H^n_+,|x|^2)}\right\}$$
and $C_1>0$ a structural constant.
\end{Theorem}
\begin{proof}
We start by noticing a direct consequence of Lemma \ref{lemprimogiro}:
\begin{equation}\label{atzero}
u(0,t)\leq \frac{M_0}{r^2_0} t \quad \text{ for all } 0<t<r_0^2,
\end{equation}
where $M_0$ is the positive constant in Lemma \ref{lemprimogiro}. This proves \eqref{linear-growth-estimate} at $x = 0$, so we need to prove the estimate at each $x_0\in\R^{2n}$ with $0<|x_0|<r_0$ . Again by Lemma \ref{lemprimogiro} we have
\begin{equation}\label{atx0sopra}
u(x_0,t)\leq \frac{M_0}{r^4_0} \left(t^2+r_0^2t \right )\leq \frac{2 M_0}{r^2_0} t \qquad \mbox{ for any } t \mbox{ such that } |x_0|^2\leq t<\sqrt{r_0^4 - |x_0|^4}.
\end{equation}
The previous upper bound (which is effective if $2|x_0|^4<r_0^4$) does not say anything about the range $t\in (0,\min\left\{|x_0|^2, \sqrt{r_0^4 - |x_0|^4}\right\})$. With this in mind, let us now denote
$$
D_{x_0}=\left\{(x,t)\in\H^n_+ \,:\, |x-x_0|< \frac{1}{10}|x_0|,\, 0<t< |x|^2\right\}.$$
We will apply a barrier argument to prove the bound \eqref{linear-growth-estimate} in $D_{x_0}$.

First, notice that, for any $(x,t)\in \overline{D_{x_0}}$ we have 
\begin{equation}\label{comparax}
\frac{9}{10}|x_0|\leq |x_0|-|x-x_0|\leq |x| \leq |x|+|x-x_0|\leq \frac{11}{10}|x_0|.
\end{equation}
In particular, using that $|x|<\frac{11}{10}r_0$, we recognize that $\overline{D_{x_0}}\subseteq B_{\frac{4}{3}r_0}(0)$. Hence, we can use Lemma \ref{lemprimogiro} once more to infer
\begin{equation}\label{ustasotto}
u(x,t)\leq \frac{M_0}{r_0^4}|x|^2(|x|^2+r_0^2)\leq \frac{\tilde{M}_0}{r_0^2}|x_0|^2 \qquad\mbox{ for any }(x,t)\in \overline{D_{x_0}}
\end{equation}
where $\tilde{M}_0=\frac{11^2(11^2+10^2)}{10^4} M_0$. Letting
$$
\alpha=\frac{25n}{8}\frac{\Lambda}{\lambda}
\quad\mbox{ and }\quad M=\max\left\{\frac{\tilde{M}_0}{1- e^{\frac{-\alpha}{100}}}, \frac{e^{\frac{61\alpha}{50}}}{4n\alpha\Lambda} r^2_0 ||f^-||_{L^\infty(D_{x_0})} \right\},
$$
we can consider the function
$$
w(z)=M \frac{|x_0|^2}{r^2_0} \left( 1- e^{-\alpha\frac{t+|x-x_0|^2}{|x_0|^2}}\right).
$$
We remark that, by \eqref{comparax}, we have
$$
\begin{cases}
w(x,t) \geq M \frac{|x_0|^2}{r^2_0} \left( 1- e^{\frac{-\alpha}{100}}\right) \qquad\,\,\,\mbox{ for } |x-x_0|=\frac{1}{10}|x_0|\mbox{ and }t\geq 0  \\
w(x,t)\geq M \frac{|x_0|^2}{r^2_0} \left( 1- e^{\frac{-81\alpha}{100}}\right) \qquad\mbox{ for } |x-x_0|\leq \frac{1}{10}|x_0|\mbox{ and }t=|x|^2.
\end{cases}
$$
Therefore,
$$
w(z)\geq \tilde{M}_0 \frac{|x_0|^2}{r^2_0}\qquad\mbox{ for any }z\in \partial D_{x_0}\cap \H^n_+.
$$
If we keep in mind \eqref{ustasotto} and the fact that $u\leq 0\leq w$ on $\{t=0\}$, the previous lower bound implies
\begin{equation}\label{ustasottoalbordo}
u\leq w\quad\mbox{ on }\partial D_{x_0}
\end{equation}
On the other hand, for any $z\in D_{x_0}$, we have
\begin{align*}
\L_A w (z)&= 
\frac{2\alpha M}{r_0^2}e^{-\alpha\frac{t+|x-x_0|^2}{|x_0|^2}}\left( \text{tr}(A(z)) -\frac{2\alpha}{|x_0|^2}\langle A(z)\left(\J x + (x-x_0)\right),\left(\J x + (x-x_0)\right) \rangle\right) \\
&\leq \frac{4\alpha M}{r_0^2}e^{-\alpha\frac{t+|x-x_0|^2}{|x_0|^2}}\left( n\Lambda -\frac{\alpha\lambda}{|x_0|^2}\left|\J x + (x-x_0)\right|^2 \right) \\
& \leq \frac{4\alpha M}{r_0^2}e^{-\alpha\frac{t+|x-x_0|^2}{|x_0|^2}}\left( n\Lambda -\frac{\alpha\lambda}{|x_0|^2}\left( |\J x| - |x-x_0|\right)^2 \right)\\
&\leq \frac{4\alpha M}{r_0^2}e^{-\alpha\frac{t+|x-x_0|^2}{|x_0|^2}}\left( n\Lambda -\frac{\alpha\lambda}{|x_0|^2}\left( \frac{9}{10}|x_0| -  \frac{1}{10}|x_0|\right)^2 \right)\\
& \leq \frac{-4n\alpha\Lambda M}{r_0^2}e^{-\alpha\frac{|x|^2+\frac{1}{100}|x_0|^2}{|x_0|^2}} \\
& \leq \frac{-4n\alpha\Lambda M}{r_0^2}e^{-\alpha\frac{61}{50}},
\end{align*}
where we used that $A(z) \in M_n(\lambda, \Lambda)$, together with the definition of $\alpha$ and \eqref{comparax}. Recalling the definition of $M$, we thus have
$$
\L_A w (z)\leq -||f^-||_{L^\infty(D_{x_0})} \leq -f^-(z)\leq f(z)\leq \L_A u (z)
 \quad \text{for all } z \in D_{x_0}.$$
Thanks also to \eqref{ustasottoalbordo}, we are then able to apply the comparison principle which yields
$$
u\leq w \quad \mbox{ in }D_{x_0}.
$$
In particular,
$$
u(x_0,t)\leq M \frac{|x_0|^2}{r_0^2}\left(1-e^{-\alpha\frac{t}{|x_0|^2}}\right)\qquad\mbox{ for any }t\mbox{ such that }0<t<|x_0|^2.
$$
The convexity of the function $\sigma\mapsto e^{-\alpha\sigma}$ (which yields $1-e^{-\alpha\sigma}\leq \alpha\sigma$ for any $\sigma\in\R$) implies
\begin{equation}\label{miracle}
u(x_0,t)\leq \frac{\alpha M}{r_0^2} t\qquad\mbox{ for any }t\mbox{ such that }0<t<|x_0|^2.
\end{equation}
The combination of \eqref{atzero}-\eqref{atx0sopra}-\eqref{miracle} yields
$$
u(z)\leq \frac{M_1}{r^2_0} t \qquad \text{ for all } \ z \in B_{r_0}(0)\cap \H^n_+
$$
if we choose
$$M_1\geq\max\left\{ 2M_0, \alpha M \right\}.$$
From a brief check of the dependence of the explicit constants and from the inequality $ ||f^-||_{L^\infty(D_{x_0})}\leq ||f^-||_{L^\infty(B_{4/3\, r_0}(0))}\leq 2 r^2_0 ||f^-||_{L^{\infty}(B_{4r_0}(0)\cap\H^n_+,|x|^2)}$, we readily realize that we can take
$$M_1=C_1\max\left\{ ||u^+||_{L^{\infty}(B_{4r_0}(0)\cap\H^n_+)}, 
\frac{r_0^4}{\lambda}||f^-||_{L^{\infty}(B_{4r_0}(0)\cap\H^n_+,|x|^2)}\right\}$$
with the choice 
$$C_1=\max\left\{ 2C_0, \frac{11^2(11^2+10^2)}{10^4} \frac{\alpha C_0}{1- e^{\frac{-\alpha}{100}}}, \frac{e^{\frac{61\alpha}{50}}}{(Q-2)\frac{\Lambda}{\lambda}} \right\}$$
where $C_0$ is the structural constant coming from Lemma \ref{lemprimogiro} and $\alpha=\frac{25}{16}(Q-2)\frac{\Lambda}{\lambda}$ as fixed above.
\end{proof}

A direct application of the previous theorem to both $u$ and $-u$ yields the following weighted $L^{\infty}$ estimate for solutions that vanish on the $\{t=0\}$ plane.

\begin{Corollary}\label{nocond}
Fix $r_0>0$. Suppose that $u\in C^2(B_{4r_0}(0)\cap\H^n_+)\cap C(\overline{B_{4r_0}(0)\cap\H^n_+})$ solves
$$\begin{cases}
\L_A u = f & \quad \text{in } B_{4r_0}(0)\cap\H^n_+, \\
u= 0 & \quad \text{on }  B_{4r_0}(0) \cap \{t=0\},
\end{cases}$$
for some $f \in L^{\infty}(B_{4r_0}(0)\cap\H^n_+,|x|^2)$. Then $\frac{u}{t}$ is bounded in $B_{r_0}(0)\cap \H^n_+$, i.e. $u\in L^\infty(B_{r_0}(0)\cap \H^n_+, t)$. In particular, we there exists a positive constant $C$ (depending only on $Q$ and $\frac{\Lambda}{\lambda}$) such that
$$
||u||_{L^{\infty}(B_{r_0}(0)\cap\H^n_+,t)} \leq \frac{C}{r_0^2} \left( ||u||_{L^{\infty}(B_{4r_0}(0)\cap\H^n_+)} + \frac{r_0^4}{\lambda}  ||f||_{L^{\infty}(B_{4r_0}(0)\cap\H^n_+,|x|^2)} \right).
$$
\end{Corollary}

\subsection{Precise Asymptotic Expansion Near the Origin}\label{LiptoC1alpha}

The results in the previous subsection show that for solutions of \eqref{half-space-BVP}, we have an estimate of the form $|u(x,t)| \leq Ct$, provided $f$ belongs to an appropriate weighted $L^{\infty}$ space; note that this is already a marked improvement from the results of Section \ref{sec:lipschitz}. In this subsection, we go further and identify the slope that determines the precise linear growth of the solution away from the $\{t = 0\}$ plane.

Our main result, Theorem \ref{ScaleInvariantHolderEstimateNormalDerivative}, shows that the normal derivative $\partial_t u(0,0)$ is well defined and that $u(x,t) - \partial_t u(0,0) t$ grows like $d^{2+\a}(x,t)$ for some $\a \in(0,1)$. To prove this, we will adapt a strategy for showing H\"older continuity of the normal derivative for solutions of uniformly elliptic second order equations in non-divergence form. Such a result was originally proved by Krylov \cite{Krylov83} and a simpler argument due to Caffarelli is presented in \cite[Chapter IV, Section 3]{KazdanBook}; see also \cite[Theorem 9.31]{GTBook} and \cite[Chapter 1, Section 1.2.3]{HanBook}.

We follow Caffarelli's approach, which entails carrying out a $C^{\a}$-type iteration argument to the function $\frac{u}{t}$. We note that, while this strategy yields a boundary $C^{1,\a}$ estimate in the classical uniformly elliptic case, the end result for solutions of \eqref{half-space-BVP} is actually a \emph{second order} Taylor expansion with respect to the metric $d$ at the origin. This reflects the fact that the $t$ variable is homogeneous of degree 2 in the metric and a H\"older estimate for $\frac{u}{t}$ at the origin translates to a growth rate of $d^{2+\a}$ for the function $u(x,t) - \partial_t u(0,0) t$. We refer to the discussion in the Introduction of this paper. 

The $C^{\a}$-type iteration for the ratio $\frac{u}{t}$ will be carried out over a family of rectangular sets in the half-space $\H^n_+$, which we now define. Fix $\delta=\min\{\frac{1}{2n+2}, \frac{\lambda}{10\Lambda}\}$. For any $r>0$, let
\begin{align*}
\r(r) := &\left\{(x,t)\,:\, 0<t<\delta r^2,\, |x|<r\right\}, \\
\r^+(r) := &\left\{(x,t)\,:\, \frac{\delta r^2}{2} <t<\delta r^2,\, |x|<r\right\}, \\
\Sigma(r) := & \,\, \partial \r(r) \cap \{t = 0\}=\left\{(x,0)\,:\, |x|\leq r\right\}.
\end{align*}
Notice that $\r(r) = \delta_r(\r(1))$.

Our goal is to construct appropriate barriers on these rectangular sets. We begin with a simple lemma.

\begin{Lemma}\label{barriers}
For any $r>0$, define the functions
\begin{align*}
\phi_1(x,t) &:=t\left( \delta + \frac{t}{r^2} -2\delta \frac{|x|^4}{r^4}\right),\\
\phi_2(t)&:=t\left(t - \delta r^2 \right).
\end{align*}
Then for any $\L_A$ with $A(z) \in M_n(\lambda,\Lambda)$, we have
\begin{align*}
\Lop_A \phi_1(x,t)  &\geq \frac{4 \lambda |x|^2}{r^2} \qquad \text{ for all } (x,t) \in \r(r),\\
\Lop_A \phi_2(x,t) & \geq 8 \lambda |x|^2 \qquad\,\,\,\, \text{for all } (x,t) \in \H^n.
\end{align*}
\end{Lemma}
\begin{proof}
By direct computation (see also \eqref{(i)}) we have 
$$
\nabla_X [t]=2 \J x, \qquad \text{and} \qquad \nabla_X[|x|^4]=4 |x|^2 x.
$$
Using again the fact that $\tr(A(x,t) \J)=0$ for any $A(\cdot)$ symmetric, we then obtain
$$
\Lop_A [t]=0,  \qquad \text{and} \qquad\Lop_A\left[|x|^4\right]=4|x|^2\tr(A(x,t))+ 8\left\langle A(x,t)x,x\right\rangle.
$$
Hence, by \eqref{ellipticity},
\begin{align*}
\Lop_A \phi_2(x,t) & = \Lop_A[t] (t - \delta r^2) + t \Lop_A [t-\delta r^2] +2  \left\langle A(x,t) \nabla_X (t), \nabla_X (t-\delta r^2) \right\rangle \\
& = 8 \left\langle A(x,t) Jx, Jx \right\rangle \geq 8\lambda|x|^2
\end{align*}
for any $(x,t)\in \H^n$. Similarly,
\begin{align*}
&\Lop_A \phi_1(x,t) \\
& = \Lop_A [t]\left( \delta + \frac{t}{r^2} -2\delta \frac{|x|^4}{r^4}\right)+ t\Lop_A\left[ \delta + \frac{t}{r^2} -2\delta \frac{|x|^4}{r^4}\right] + 2 \left\langle A(x,t)\nabla_X (t), \nabla_X \left( \delta + \frac{t}{r^2} -2\delta \frac{|x|^4}{r^4}\right)\right\rangle\\
&=\frac{8}{r^2}\left\langle A(x,t)Jx,Jx\right\rangle - \frac{32\delta |x|^2}{r^4}\left\langle A(x,t)Jx,x\right\rangle -\frac{8\delta t}{r^4}\left(|x|^2\tr (A(x,t))+ 2\left\langle A(x,t)x,x\right\rangle\right)\\
&\geq \frac{8\lambda |x|^2}{r^2} - \frac{32\delta \Lambda |x|^4}{r^4}-\frac{16\delta \Lambda t |x|^2}{r^4}(n+1)\\
&=\frac{8\lambda |x|^2}{r^2}\left(1- \frac{4\delta \Lambda |x|^2}{\lambda r^2} - (2n+2)\frac{\delta \Lambda t}{\lambda r^2}\right).
\end{align*}
If $(x,t)\in \r(r)$, then
$$
\Lop_A \phi_1 (x,t)\geq \frac{8\lambda |x|^2}{r^2} \left(1- 4\delta \frac{\Lambda}{\lambda} - (2n+2)\delta^2 \frac{\Lambda }{\lambda}\right).
$$
The desired inequality for $\Lop_A \phi_1$ follows from our choice of $\delta$.
\end{proof} 

We will use the barrier functions from the previous lemma to estimate the infimum of $v =\frac{u}{t}$.

\begin{Proposition}\label{firstestimate}
Fix $r>0$. Suppose $w \in C^2(\r(r))\cap C(\overline{\r(r)})$ is non-negative and satisfies $\Lop_A w \leq f$ in $\r(r)$ for $f \in L^{\infty}(\r(r),|x|^2)$. Then the function $v:= \frac{w}{t}$ satisfies
\begin{equation}\label{firstestimate1}
\inf_{\r^+(r)} v \leq \frac{2}{\delta} \inf_{\r(\frac{r}{2})} v + \frac{r^2}{\lambda}||f^+||_{L^{\infty}(\r(r),|x|^2)}.
\end{equation}
\end{Proposition}
\begin{proof}
Assume $\inf_{\r(\frac{r}{2})} v<+\infty$ (otherwise there is nothing to prove). Set 
$$m = \inf\{v(x,\delta r^2) : |x| < r \} \qquad \text{and} \qquad F = ||f^+||_{L^{\infty}(\r(r),|x|^2)}.$$
Note that $m \geq 0$, and that $w(x,\delta r^2) = \delta r^2 v(x,\delta r^2) \geq m \delta r^2$ for all $|x| < r$. Let $\phi_1$ and $\phi_2$ be as in Lemma \ref{barriers}, and consider the function
$$\beta(x,t) := m \phi_1(x,t) + \frac{F}{8 \lambda} \phi_2(t) = m \delta t\left(1 + \frac{t}{\delta r^2} -\frac{2|x|^4}{r^4}\right) +  \frac{F t}{8 \lambda} \left(t - \delta r^2 \right). $$
We note that $\beta$ satisfies the following properties:
\begin{enumerate}
\item[(i)] $\beta(x,0) = 0 \leq w(x,0)$ for all $x \in \Sigma(r)$;
\item[(ii)] $\beta(x,t) \leq 0 \leq w(x,t)$ for all $|x| = r$ and $0 < t \leq \delta r^2$;
\item[(iii)] For all $|x|< r$,
\[\beta(x,\delta r^2) \leq 2 m \delta^2 r^2 < m \delta r^2 \leq w(x,\delta r^2),\]
where the second inequality holds because $\delta < \frac{1}{2}$;
\item[(iv)] For any $(x,t) \in \r(r)$, we have by Lemma \ref{barriers}
\begin{align*}
\Lop_A \beta (x,t) & = m \Lop_A \phi_1 (x,t) + \frac{F}{8\lambda} \Lop_A \phi_2 (x,t)\\
& \geq \frac{F}{8\lambda} \Lop_A \phi_2 (x,t) \\
& \geq F |x|^2\\
& \geq f(x,t) \geq \Lop_A w(x,t).
\end{align*}
\end{enumerate}
It follows from the comparison principle that $\beta \leq w$ on $\r(r)$, which we can rewrite as
$$
v(x,t)=\frac{w(x,t)}{t} \geq \frac{\beta(x,t)}{t} = m \delta \left(1 + \frac{t}{\delta r^2} -\frac{2|x|^4}{r^4}\right) +  \frac{F}{8 \lambda} \left(t - \delta r^2 \right) 
$$
for any $t>0$. Restricting to $(x,t) \in \r(\frac{r}{2}) \subset \r(r)$, we deduce that, since $|x| < \frac{r}{2}$ and $t > 0$ in $\r(\frac{r}{2})$, 
$$
 v(x,t) +  \frac{F \delta r^2}{8 \lambda}  \geq \frac{7m \delta}{8}.
$$
Since $m \geq \inf_{ \r^+(r)} v$, it follows that 
$$
\inf_{\r^+(r)} v \leq \frac{8}{7\delta} \inf_{\r(\frac{r}{2})} v + \frac{r^2 F}{7\lambda},
$$
which immediately implies \eqref{firstestimate1}.
\end{proof}

Since, on the sets $\r^+(r)$,  we are a positive distance away from the boundary $\{t=0\}$, we can obtain pointwise estimates for $v=\frac{u}{t}$ from the interior Harnack inequality established in Theorem \ref{thm:Harnack}. This is the content of the next proposition, which is the only place where we invoke the assumption \eqref{CordesLandis}.

\begin{Proposition}\label{secondestimate}
Assume \eqref{CordesLandis} holds. There exist structural constants $C> 1$ and $K>1$ such that, if $r>0$ and $w$ is a non-negative $C^2$-solution of $\Lop_A w = f$ in $\r(Kr)$ for $f \in L^{\infty}(\r(Kr),|x|^2)$, then the function $v := \frac{w}{t}$ satisfies
\[
\sup_{\r^+(r)} v \leq C \left(\inf_{\r^+(r)} v + r^2  ||f||_{L^{\infty}(\r(Kr),|x|^2)}  \right).
\]
\end{Proposition}
\begin{proof}
We claim there exist structural constants $\tilde{C}_H, K >1$ such that
\begin{equation}\label{Harnackcon2rett}
\sup_{\r^+(r)} w \leq \tilde{C}_H\left(\inf_{\r^+(r)} w + r^2 ||f||_{L^{\infty}(\r(Kr))}  \right)
\end{equation}
for any $r>0$. This basically follows from the inhomogeneous Harnack inequality in Theorem \ref{thm:Harnack} combined with a covering argument that allows us to rewrite the estimate in terms of the cylindrical sets $\r(r)$. Once we have \eqref{Harnackcon2rett}, it suffices to notice that, since $\frac{\delta}{2} r^2<t<\delta r^2$ in $\r^+(r)$, we have
$$\frac{\delta}{2} r^2 \sup_{\r^+(r)} v \leq \sup_{\r^+(r)} w \qquad \text{and} \qquad \inf_{\r^+(r)} w \leq \delta r^2 \inf_{\r^+(r)} v,$$
and so
\begin{align*}
\sup_{\r^+(r)} v&\leq \frac{2}{\delta r^2} \sup_{\r^+(r)} w \leq \tilde{C}_H\left(\frac{2}{\delta r^2} \inf_{\r^+(r)} w + \frac{2}{\delta} ||f||_{L^{\infty}(\r(Kr))}  \right)\\
&\leq \tilde{C}_H\left( 2 \inf_{\r^+(r)} v + \frac{2}{\delta} ||f||_{L^{\infty}(\r(Kr))}\right) \\
&  \leq  \tilde{C}_H\left( 2 \inf_{\r^+(r)} v + \frac{2K^2 r^2}{\delta} ||f||_{L^{\infty}(\r(Kr),|x|^2)} \right)
\end{align*}
where in the final inequality we used that $||f||_{L^{\infty}(\r(Kr))}\leq (Kr)^2 ||f||_{L^{\infty}(\r(Kr),|x|^2)}$. This shows the desired Harnack-type inequality for $v$ as in the statement of the Proposition with the choice $C=\frac{2K^2}{\delta} \tilde{C}_H$.

For the reader's convenience, we provide some details of the covering argument used to verify \eqref{Harnackcon2rett}. First notice that, by scale-invariance with respect to the family of dilations $\delta_r$, it is enough to show \eqref{Harnackcon2rett} for $r=1$. Since the compact set $\overline{\r^+(1)}$ is at positive distance from the half-space $\{(\xi,\tau)\,:\, \tau\leq 0\}$, we know there exists a positive constant $\rho$ (depending on $\delta$) such that
$$
B_\rho(z)\subset \H^n_+\quad \quad \text{for all } z\in \overline{\r^+(1)}.
$$
Hence, we can now choose $K>1$ (depending on $\delta, \rho$) such that
\begin{equation}\label{covdentrorett}
B_\rho(z)\subseteq \r(K) \quad \text{for all } z\in \overline{\r^+(1)}.
\end{equation}
Moreover, for the constant $K_H$ in Theorem \ref{thm:Harnack}, we consider the following open covering of $\overline{\r^+(1)}$ $$\left\{ B_{\frac{\rho}{K_H}}(z)\,:\, z\in \overline{\r^+(1)} \right\}.$$ 
By compactness, there exist $p\in \N$ and $z_1,\ldots, z_p \in \overline{\r^+(1)} $ such that $\r^+(1)\subseteq \bigcup_{i=1}^p B_{\frac{\rho}{K_H}}(z_i)$. We now apply the Harnack inequality, Theorem \ref{thm:Harnack}, on the balls $B_{\frac{\rho}{K_H}}(z_i)$. More precisely, let $u$ be a non-negative solution to $\Lop_A uw = f$ in $\r(K)$ for $f \in L^{\infty}(\r(K),|x|^2)\subseteq L^{\infty}(\r(K))$ and let $C_H>1$ denote the Harnack constant in Theorem  \ref{thm:Harnack}. Then, thanks to \eqref{covdentrorett}, we have
\begin{align*}
w(z) & \leq C_H\left(w(\zeta) + ||f||_{L^{\infty}(B_{\rho}(z_i))}\right)\\
& \leq C_H\left( w(\zeta) + ||f||_{L^{\infty}(\r(K))}\right)\quad \text{ for all } z,\zeta\in B_{\frac{\rho}{K_H}}(z_i), \ i= 1,\ldots,p.
\end{align*}
By applying the previous inequality a finite number of times, we infer \eqref{Harnackcon2rett} for $r=1$. The constant $\tilde{C}_H$ can be taken as $\tilde{C}_H=C_H(C_H+1)^{p-1}$ (we stress that $p$ and $C_H$ depend just on $\rho, K_H, \delta, Q, \Lambda, \lambda$). Indeed, if we pick any two points in $\r^+(1)$ we can consider the (Euclidean) segment connecting these points. By the convexity of $\r^+(1)$ and the covering property of $\bigcup_{i=1}^p B_{\frac{\rho}{K_H}}(z_i)$, such a segment is still contained in $\r^+(1)$ and we can look at the number $q\in\N$ of balls of the type $B_{\frac{\rho}{K_H}}(z_i)$ that intersect the segment; in our case, we always have $q\leq p$ since the balls $B_{\frac{\rho}{K_H}}(z_i)$ are (Euclidean) convex. We therefore need to apply the inequality displayed above at most $p$ times to conclude the proof.
\end{proof}

We are now ready to prove the main result of this section.

\begin{Theorem}\label{ScaleInvariantHolderEstimateNormalDerivative}
Assume \eqref{CordesLandis} holds. Suppose $u \in C^2(B_{4}(0)\cap \H^n_+) \cap C(\overline{B_{4}(0)\cap \H^n_+})$ solves
\[
\begin{cases}
\Lop_A u = f & \quad \text{in } B_{4}(0)\cap \H^n_+, \\
u = 0 & \quad \text{on } B_{4}(0)\cap \{t=0\},
\end{cases}
\]
for some $f \in L^{\infty}(B_{4}(0)\cap \H^n_+,|x|^2)$. Then $\partial_t u(0,0)$ exists. Moreover, there exist constants $C > 1$ and $\rho_0, \a \in(0,1)$ depending only on $Q, \lambda, \Lambda$ such that for all $z \in B_{\rho_0}(0)\cap \H^n_+$
\[
|u(z) - \partial_t u(0,0) t| \leq C\left( ||u||_{L^{\infty}(B_{4}(0)\cap \H^n_+)}  + ||f||_{L^{\infty}(B_{4}(0)\cap \H^n_+,|x|^{2})} \right) d^{2+\alpha}(z,(0,0)).
\]
\end{Theorem}

\begin{Remark}
As the proof below will make evident, Theorem \ref{ScaleInvariantHolderEstimateNormalDerivative} relies on the Cordes-Landis assumption \eqref{CordesLandis} only via Proposition \ref{secondestimate} and, even more indirectly, on the Harnack inequality, Theorem \ref{thm:Harnack}. If the coefficients of $\L_A$ are more regular (say $d$-H\"older continuous), then one can use the Harnack inequality from \cite{Bonfiglioli-Uguzzoni-07} instead.
\end{Remark}

\begin{proof}[Proof of Theorem \ref{ScaleInvariantHolderEstimateNormalDerivative}] 
Let $K>1$ be the constant provided by Proposition \ref{secondestimate}, and fix
$$
\rho= \frac{1}{2K} \left(\frac{1}{1+\delta^2}\right)^{\frac{1}{4}}.
$$
We also denote
$$
v := \frac{u}{t}.
$$
By Corollary \ref{nocond} we know that $v\in L^\infty(B_{1}(0)\cap \H^n_+)$. Keeping in mind the definition of $\rho$ and of the sets $\r(\cdot)$, we have  $\r(2K\rho)\subseteq B_{1}(0)\cap \H^n_+$ and $v\in L^\infty(\r(2K\rho))$ with the estimate
\begin{equation}\label{usolip}
||v||_{L^\infty(\r(2K\rho))}\leq
C_1 \left( ||u||_{L^{\infty}(B_{4}(0)\cap \H^n_+)} + \frac{1 }{ \lambda}  ||f||_{L^{\infty}(B_{4}(0)\cap \H^n_+,|x|^2)} \right),
\end{equation}
where $C_1>1$ is the constant (named $C$) in Corollary \ref{nocond}. We can thus define, for any $r \in (0,2K\rho]$,
\[m(r) := \inf_{\r(r)} v, \qquad\mbox{and} \qquad M(r) := \sup_{\r(r)} v.\]
Fix an arbitrary $r\in (0,\rho]$ and consider the function 
\[w_- = u - m(2Kr) t,\]
which is non-negative and solves $\L_Aw_- = f$ in $\r(2Kr)$. Applying Proposition \ref{secondestimate} to $w_-$ (let us call $C_2>1$ the constant named $C$ in Proposition \ref{secondestimate}), we have
$$\sup_{\r^+(2r)} (v-m(2Kr)) \leq C_2 \left( \inf_{\r^+(2r)} (v-m(2Kr)) + 4r^2   ||f||_{L^{\infty}(\r(2Kr),|x|^{2})}   \right).$$
We can also apply Proposition \ref{firstestimate} to $w_-$ to obtain
\begin{align*}
\inf_{\r^+(2r)} (v-m(2Kr)) &  \leq \frac{2}{\delta} \inf_{\r(r)} (v-m(2Kr)) + \frac{4r^2}{\lambda} ||f||_{L^{\infty}(\r(2r),|x|^{2})} \\
& = \frac{2}{\delta} \left[ m(r)-m(2Kr) \right] + \frac{4r^2}{\lambda} ||f||_{L^{\infty}(\r(2Kr),|x|^{2})}. 
\end{align*}
Substituting into the previous estimate, we obtain
\begin{equation}\label{first-inequality-for-oscillation}
    \sup_{\r^+(2r)} (v-m(2Kr)) \leq \tilde{C}_2\left( m(r)-m(2Kr) + r^2 ||f||_{L^{\infty}(\r(2Kr),|x|^{2})}  \right),
\end{equation}
where $\tilde{C}_2=\frac{2C_2}{\delta}\max\{1, \frac{2\delta(1+\lambda)}{\lambda}\}>1$. 

Carrying out similar arguments with the function 
\[w_+ = M(2Kr) t - u,\]
which is non-negative and solves $\L_Aw_+ = -f$ in $\r(2Kr)$, we obtain
\begin{equation}\label{second-inequality-for-oscillation}
\sup_{\r^+(2r)} (M(2Kr)-v) \leq \tilde{C}_2 \left(M(2Kr) - M(r) + r^2  ||f||_{L^{\infty}(\r(2Kr),|x|^{2})}  \right).
\end{equation}
Adding the inequalities \eqref{first-inequality-for-oscillation} and \eqref{second-inequality-for-oscillation} and using $\underset{\r^+(2r)}{\osc} v \geq 0$, we find that
\[
M(2Kr)-m(2Kr) \leq \tilde{C}_2 \left(\left[M(2Kr) - m(2Kr) \right] - \left[M(r) - m(r) \right] + 2r^2  ||f||_{L^{\infty}(\r(2Kr),|x|^{2})} \right).
\]
Setting $$\gamma := \frac{\tilde{C}_2-1}{\tilde{C}_2} \in (0,1),$$ we obtain the following decay of oscillation for the function $v$:
$$\underset{\r(r)}{\osc} v \leq \gamma \ \underset{\r(2Kr)}{\osc} v +  2 r^2  ||f||_{L^{\infty}(\r(2Kr),|x|^{2})}\quad\mbox{ for all }0<r\leq \rho.
$$
A standard iteration argument (see, for instance, \cite[Lemma 8.23]{GTBook}) implies that for some universal constants $\bar{C} > 1$ and $\alpha \in (0,1)$, we have
$$\underset{\r(r)}{\osc} v \leq \bar{C} r^{\alpha} \left( ||v||_{L^{\infty}(\r(2K\rho))}  +  ||f||_{L^{\infty}(\r(2K\rho),|x|^{2})}\right).
$$
Together with \eqref{usolip}, this yields, for some universal constant $C>1$,
\begin{equation}\label{bellobello}
\underset{\r(r)}{\osc} \left(\frac{u}{t}\right) \leq C r^{\alpha} \left( ||u||_{L^{\infty}(B_{4}(0)\cap \H^n_+)}  +  ||f||_{L^{\infty}(B_{4}(0)\cap \H^n_+,|x|^{2})}\right) \quad\mbox{for all }0<r\leq \rho.
\end{equation}
Since $|v(z)-v(z')|\leq \underset{\r(r)}{\osc}v$ for any $z,z'\in \r(r)$, by the Cauchy criterion for existence of limits, and keeping in mind that $u(0,0)=0$, we realize that
$$
\exists \lim_{\substack{(x,t)\to (0,0) \\ (x,t)\in \H^n_+}} \frac{u(x,t)}{t}= \lim_{\substack{(x,t)\to (0,0) \\ (x,t)\in \H^n_+}} \frac{u(x,t)-u(0,0)}{t}=:\partial_t u(0,0)\in\R.
$$
Therefore, we deduce from \eqref{bellobello} that for all $(x,t)\in \r(r)$ and $0 < r \leq \rho$,
\[
\bigg|\frac{u(x,t)}{t} -  \partial_t u(0,0) \bigg| \leq C r^{\alpha}  \left( ||u||_{L^{\infty}(B_{4}(0)\cap \H^n_+)}  +  ||f||_{L^{\infty}(B_{4}(0)\cap \H^n_+,|x|^{2})}\right). 
\]
If we set $\rho_0=\sqrt{\delta} \rho$, then for every $z\in B_{\rho_0}(0)\cap \H^n_+$ we have $z\in \r\left(\frac{d(z,0)}{\sqrt{\delta}}\right)$ and $\frac{d(z,0)}{\sqrt{\delta}}<\rho$. It follows that for all $z\in B_{\rho_0}(0)\cap \H^n_+$
$$
\left|u(z) -  \partial_t u(0,0) t \right| \leq \frac{C}{\delta^{\frac{\alpha}{2}}}\, t\, d^{\alpha}(z,(0,0))  \left( ||u||_{L^{\infty}(B_{4}(0)\cap \H^n_+)}  +  ||f||_{L^{\infty}(B_{4}(0)\cap \H^n_+,|x|^{2})}\right).
$$
Since $t\leq d^2((x,t),(0,0))$, this completes the proof of the theorem.
\end{proof}

As mentioned in the introduction, Theorem \ref{ScaleInvariantHolderEstimateNormalDerivative} implies $u$ separates from its intrinsic second order Taylor polynomial $T_2u$ at a rate of $d^{2+\a}$, and so can be thought of as a $C^{2,\a}$ type estimate.
It is also remarkable that although the function $u$ is initially assumed to be only continuous up to the boundary $\{t=0\}$, it ends up having second order differentiability properties at the origin.

It is known that the order $2$ is critical for the regularity at $(0,0)$ of solutions in the half-space $\H^n_+$, which can be seen as a by-product of the fact that the linear function $\bar{u}(x,t)=t$ is $\L_A$-harmonic and homogeneous of degree $2$; we refer the interested reader to the discussions in \cite[Remarque on pg. 106]{Gaveau} and \cite[pg. 235]{Jerison81-2}. Jerison initiated in \cite{Jerison81-2} a thorough analysis of the critical degree for the regularity of the solutions to $\Delta_X u=0$ at characteristic boundary points in scale-invariant domains like $\{t>M |x|^2\}$. In the case of the half-space $\H^n_+$, in \cite[Section 5]{Jerison81-2}, Jerison shows that the solutions of the Dirichlet problem 
\[
\begin{cases}
\Delta_X u = 0 & \quad \text{in } B_1(0) \cap \H^n_+\\
u = g & \quad \text{in } B_1(0) \cap \{t =0\}
\end{cases}
\]
for generic $C^2$ boundary data $g$ may fail to have second order regularity around $(0,0)$. We can use Theorem \ref{ScaleInvariantHolderEstimateNormalDerivative} to identify a class of boundary data $g$ for which the solution $u$ of the Dirichlet problem is guaranteed to enjoy second order estimates.

\begin{Corollary}\label{ScaleInvariantHolderEstimateNormalDerivativeWithDirichletData}
Assume $g: \{x\in\R^{2n}\,:\, |x|\leq 4\}\to \R$ is a $C^2$-function such 
that  
\begin{equation}\label{Dirichlet-data-class}
\Delta g \in L^{\infty}(\{|x|<4\},|x|^2).
\end{equation}
Suppose $u\in C^2(B_{4}(0)\cap \H^n_+) \cap C(\overline{B_{4}(0)\cap \H^n_+})$ solves
\[
\begin{cases}
\Delta_X u = 0 & \quad \text{in } B_{4}(0)\cap \H^n_+, \\
u = g & \quad \text{on }B_{4}(0)\cap \{t=0\}.
\end{cases}
\]
Then $\partial_t u(0,0)$ exists. Moreover, there exist constants $C > 1$, $\rho_0\in(0,1)$, and $\alpha \in (0,1)$ depending only on $Q$ such that
\[
\frac{|u(x,t) - g(x) - \partial_t u(0,0) t|}{d^{2+\alpha}((x,t),(0,0)) } \leq C\left( ||u||_{L^{\infty}(B_{4}(0)\cap \H^n_+)}+ ||g||_{L^{\infty}(\{|x|<4\})}  + ||\Delta g||_{L^{\infty}(\{|x|<4\},|x|^{2})} \right)
\]
for all $(x,t) \in B_{\rho_0}(0)\cap \H^n_+$.
\end{Corollary}
\begin{proof}
Consider the function $w(x,t) = u(x,t) - g(x)$. Then $w$ solves
\[
\begin{cases}
\Delta_X w = -\Delta g & \quad \text{in } B_{4}(0)\cap \H^n_+, \\
w = 0 & \quad \text{on }B_{4}(0)\cap \{t=0\}.
\end{cases}
\]
The assumption \eqref{Dirichlet-data-class} allows us to apply Theorem \ref{ScaleInvariantHolderEstimateNormalDerivative} with $f(x,t)=-\Delta g(x)$ and the conclusion follows in a straightforward way since $\partial_t u=\partial_t w$.
\end{proof}

Expanding our results above to encompass bounded source terms $f$ that reside outside the $L^{\infty}(\cdot,|x|^2)$ class promises to be challenging because of the following concrete counterexample.
\begin{Example}\label{example}
Recall the function $\phi(x,t) = |x|^4 + t^2$ from \eqref{defPhiPsi}. Fix $\e> 0$. For $q \in (0,1)$ define
\[u_{\e}(x,t) := t \phi^q(x,t+\e)  = t d^{4q}((x,t), (0,-\e)).\]
We have
\[u_{\e}(x,0) = 0 \quad \text{and} \quad \partial_t u_{\e}(0,0) = \phi(0,\e)^q =\e^{2q}.
\]
A straightforward, but tedious calculation using \eqref{(i)}, \eqref{(ii)}, and translation invariance shows
\[\Delta_X u_{\e}(x,t) = 8 q\phi^{q-1}(x,t+\e) |x|^2 \left[t (2q+n) +2(t+\e)\right]\]
Since
\begin{align*}
8 q\phi^{q-1}(x,t+\e) |x|^2 \left[t (2q+n) +2(t+\e)\right] & \leq 8 q\phi^{q-1}(x,t+\e) |x|^2 (2q+n+2)(t+\e) \\
& \leq 4q(2q+n+2) \phi^q(x,t+\e),
\end{align*}
we conclude that the standard $L^{\infty}$ norm $||\Delta_X u_{\e}||_{L^{\infty}(B_4(0)\cap\H^n_+)}$ is bounded uniformly in $\e>0$ (and in $q\in (0,1)$), whereas the weighted $L^{\infty}$ norm $||\Delta_X u_{\e}||_{L^{\infty}(B_4(0)\cap\H^n_+, |x|^2)}$ becomes unbounded as $\e \to 0^+$ (for any fixed $q\in (0,\frac{1}{2})$).
Meanwhile,
\[
\left|\frac{u_{\e}(0,t)}{t}-\partial_t u_{\e}(0,0)\right| = |(t+\e)^{2q} - \e^{2q}|.
\]
It follows that, for any $C>0$ and $\a\in (0,1)$, we can pick $q < \a/4$ and make the previous expression bigger than $C d^{\a}((0,t),(0,0))=C t^{\a/2}$ for some positive $\e$ and $t$ small enough. This shows that we cannot substitute $||f||_{L^{\infty}(B_{4}(0)\cap \H^n_+,|x|^{2})}$ with $||f||_{L^{\infty}(B_{4}(0)\cap \H^n_+)}$ in Theorem \ref{ScaleInvariantHolderEstimateNormalDerivative}.
\end{Example}

The above discussion suggests that tackling higher order boundary regularity in general domains with characteristic points will be a delicate task. We plan to address this problem in a future work by restricting our attention to suitable perturbations of the flat scenario studied in this paper.


\bibliographystyle{amsplain}
\bibliography{AbedinTralli}


\end{document}